\documentclass[10pt]{article}
\usepackage{times}
\usepackage{amsmath,amsfonts,amstext,amssymb,amsbsy,amsopn,amsthm,eucal,slashed,amstext,enumerate,accents}
\usepackage{color}
\usepackage{txfonts}
\usepackage{dsfont}
\usepackage{hyperref}
\usepackage{graphicx}   
\usepackage[colorinlistoftodos]{todonotes}

\numberwithin{equation}{section}
  \theoremstyle{plain}
 \newtheorem{theorem}{Theorem}
 \newtheorem{propositionletter}{Proposition}
  \newtheorem{theoremletter}{Theorem}

\newtheorem{proposition}{Proposition}

 \newtheorem{lemma}{Lemma}
 
\newtheorem*{conjecture}{Conjecture}
\newtheorem*{claim}{\indent Claim}

 \newtheorem{remark}{Remark}

\theoremstyle{definition}

\newtheorem{example}{Example}

\def \no{\nonumber}

\renewcommand{\div}{{\rm div}}

\newcommand{\pa}{\partial}

\newcommand{\Sn}{\mathbb{S}^n}
\newcommand{\Sp}{\mathbb{S}}
\newcommand{\N}{\mathbb{N}}

\newcommand{\tr}{\text{tr}}

\newcommand{\RR}{\mathbb{R}}
\newcommand{\Rn}{\mathbb{R}^n}

\newcommand{\ud}{\mathrm{d}}

\allowdisplaybreaks

\title{Green function rigidity
and the mass of  hypersurfaces under inversion }
\author{Xuezhang Chen\thanks{X. Chen and J. Gan are partially supported by NSFC (No.11771204). Email: xuezhangchen@nju.edu.cn.},~ Jiaxue Gan\thanks{J. Gan's email: 602025210004@smail.nju.edu.cn.}~ and Yalong Shi\thanks{Y. Shi is partially supported by NSFC (No.12371058). Email:shiyl@nju.edu.cn.}\\
		{\small $^{\ast}$$^{\dag}$$^{\ddag}$School of Mathematics \& IMS, Nanjing University, Nanjing 210093, P.R. China }}
\date{}

\begin{document}
\maketitle

\begin{abstract}
This is a sequel to \cite{Chen-Shi}. We prove the Green function rigidity conjecture in \cite{Chen-Shi} for conformal Laplacian in dimension $n\geq 3$. For the Paneitz operator, we prove the Green function rigidity conjecture when $n\neq 4k+2, k\geq 2$. Important ingredients in our proof are the positive mass theorem and the positive energy theorem for Paneitz operator. As a byproduct, we also obtain a new formula for the ADM mass of an asymptotically flat hypersurface that allows for a non-entire graph.

\medskip

\end{abstract}
\medskip
{\footnotesize \tableofcontents}

\section{Introduction}

A complete Riemannian manifold $(N^n,g)$ is called asymptotically flat of asymptotic order $\tau>0$ and regularity order $\ell\in\N$, if there exists a compact set $K\subset N$ such that  $N\setminus K$ is diffeomorphic to $ \mathbb{R}^n \backslash B_R(0)$ for some $R>0$, and under this coordinate chart, $g-\delta \in O_{\ell}(|y|^{-\tau})$ as $|y| \to \infty$. The coordinates $y$ are called asymptotic coordinates. Here, $g-\delta \in O_{\ell}(|y|^{-\tau})$ means that 
$$ |\pa^k(g_{\alpha \beta}-\delta_{\alpha \beta})|=O(|y|^{-\tau-k}), \quad \forall~0 \leq k \leq \ell.$$
In the literature, the regularity order $\ell$ is usually chosen to be $2$. Accordingly, when we go to study the fourth order energy for the Paneitz operator, the regularity order requires to be  $4$.

The ADM mass of an asymptotically flat (AF) manifold $(N^n,g)$ was defined by Arnowitt-Deser-Misner \cite{ADM} in the 1960s in the context of general relativity: 
\begin{equation}\label{Def:ADM_mass}
m(g):=\frac{1}{|\Sp^{n-1}|}\lim_{t \to \infty} \int_{\Sp_t^{n-1}}g^{\beta \gamma}(\pa_\gamma g_{\alpha \beta}-\pa_\alpha g_{\beta \gamma})\nu^\alpha \ud\sigma
\end{equation}
if the limit exists. Here, $\nu$ is the outward unit normal vector field on the sphere $\Sp_t^{n-1}:=\{|y|=t\}$ with respect to the Euclidean metric. In \cite{bartnik}, Bartnik proved that when the asymptotic order $\tau>\frac{n-2}{2}$, regularity order $\ell=2$ and the scalar curvature is integrable, the mass for AF manifold is independent of the choice of asymptotic coordinates and is thus a geometric invariant for $(N,g)$. See Lee-Parker \cite{Lee-Parker} for a good introduction to the basics of ADM mass.

In \cite{SchYau79,Schoen87}, R. Schoen and S.T. Yau proved the Positive Mass Theorem (PMT):{\it\   Assume that $(N^n,g), 3\leq n\leq 7$ is an asymptotically flat manifold of asymptotic order $\tau>\frac{n-2}{2}$ (regularity order at least 2), and $0 \leq R_g \in L^1(N,g)$. Then the mass $m(g) \geq 0$, with equality if and only if $(N,g)$ is isometric to the flat $\RR^n$.} 

When $N$ is a spin manifold, an alternative proof is given by E. Witten \cite{Witten} in dimension 3 (see also \cite{ParkerTaubes}),  which works for all dimension under the spin assumption, see \cite{bartnik}.

The Positive Mass Theorem plays an important role in the study of first and third named authors \cite{Chen-Shi} of the Green function for the conformal Laplacian on closed Euclidean hypersurfaces.  Based on the belief that the explicit representation formula of the Green function for certain conformally invariant linear (pseudo-)differential operator should determine the geometry and topology of the manifold, the  authors \cite{Chen-Shi} raised the following \emph{Green Function Rigidity Conjecture (\textbf{GFRC})}.
\begin{conjecture}
    Let $M^n\subset\RR^{n+1}$ be a closed embedded hypersurface with induced metric $g$. Assume that for some $k\in\mathbb{Z}_+$, the Green function $G$ for the GJMS operator $P_{2k}^g$ exists. Suppose one of the following conditions holds: 
    \begin{enumerate}[(1)]
        \item $2k=n$ and  {\bf for some}  $Q\in M$, $G(\cdot, Q)$ is of the form $-c_n\log\|\cdot-Q\|+c$, where  $c \in \RR$ and $c_n=\frac{2}{(n-1)!|\Sn|}$;
        \item $2k<n$ or  $2k>n$ when $n$ is odd, and {\bf for some}  $Q\in M$, $G(\cdot, Q)$ is of the form $c_{n,k}\|\cdot-Q\|^{2k-n}$, where $c_{n,k}=\frac{\Gamma(\frac{n}{2}-k)}{2^{2k}\pi^{\frac{n}{2}}\Gamma(k)}$.
    \end{enumerate}
    Then $(M,g)$ is a round sphere. Here and throughout, let $\|\cdot\|$ denote the Euclidean norm. 
\end{conjecture}
Recall that the GJMS operator $P_{2k}^g$ for $k \in \mathbb{Z}_+$ on a Riemannian manifold $(M^n,g)$ is a conformally invariant linear differential operator whose principal part agrees with $(-\Delta_g)^k$, named after Graham, Jenne, Mason and Sparling \cite{GJMS}. When $k=1$, $P_2^g$ is precisely the conformal Laplacian:
$$P_2^g=-\Delta_g+\frac{n-2}{4(n-1)}R_g,\quad n\geq 2.$$
When $n=2$, it reduces to the Laplacian. 

The operator $P_4^g$ was discovered earlier by Paneitz \cite{Paneitz} in $1983$ on a smooth Riemmanian manifold $(M^n,g),n \geq 3$, and hence also referred to as `Paneitz operator':
$$P_4^g=\Delta^2_g+\frac{4}{n-2}\div\big(\sum_\alpha \mathrm{Ric}(\nabla^g\cdot,e_\alpha)e_\alpha\big)-\frac{n^2-4n+8}{2(n-1)(n-2)}\div(R_g\nabla^g\cdot)+\frac{n-4}{2}Q_g,$$
where $Q_g$ is Branson's $Q$-curvature, defined by
\begin{equation}\label{def:Q-curv}
Q_g=-\Delta_g \sigma_1(A_g)+4 \sigma_2(A_g)+\frac{n-4}{2}\sigma_1(A_g)^2.
\end{equation}
 Here
$$A_g=\frac{1}{n-2}\Big(\mathrm{Ric}_g -\frac{R_g}{2(n-1)}g\Big)$$
is the Schouten tensor and $\sigma_k(A_g)$ is the $k$-th elementary symmetric  function of $g^{-1}A_g$.

For more about the GJMS operators and their applications, readers are referred to \cite{CG} for a recent survey.

For the conformal Laplacian and Paneitz operator, the authors \cite{Chen-Shi} carried out the first step towards the above conjecture.  For example, they proved that the pole $Q$ of Green function  must be an umbilical point, and succeeded in verifying \emph{\textbf{GFRC}} for the conformal Laplacian in dimensions $n=3,4,5$. 

The proof  of \emph{\textbf{GFRC}}  for the conformal Laplacian  in \cite{Chen-Shi} goes as follows. First, the familiar inverted coordinates happen to provide asymptotic coordinates of order $2$ for $(M\backslash\{Q\}, \hat g)$, where $\hat g=G(\cdot,Q)^{\frac{4}{n-2}}g$, such that when $n=3,4,5$, the mass for $(M\backslash\{Q\}, \hat g)$ is well defined. Then a  direct computation yields that the mass of $(M\backslash\{Q\}, \hat g)$ vanishes. Since the scalar curvature of $\hat g$ is zero, the Positive Mass Theorem guarantees that $(M\backslash\{Q\}, \hat g)$ is isometric to the flat Euclidean space. Observe that $\hat g=(c_{n,2})^{\frac{4}{n-2}} \|\cdot-Q\|^{-4}g$ is essentially the induced metric of $I_Q(M\backslash\{Q\})\subset \RR^{n+1}$, where 
\begin{equation}\label{def:inversion}
I_Q(z)=Q+\frac{z-Q}{\|z-Q\|^2},\quad \forall~z\in \RR^{n+1}\backslash\{Q\}
\end{equation}
 is the inversion of $\RR^{n+1}$ with respect to $Q$. Thus $I_Q(M\backslash\{Q\})$ is a complete flat hypersurface. Hartman-Nirenberg's theorem in \cite{Hartman-Nirenberg} states that this hypersurface must be a generalized cylinder in $\RR^{n+1}$. This gives an almost explicit parametrization of $I_Q(M\backslash\{Q\})$, hence also of $M\backslash\{Q\}$. Combining this with the fact that $Q$ is umbilical forces such a generalized cylinder to be a hyperplane. Hence, $(M,g)$ is a round sphere. 

In this paper, we extend the above result to all dimensions $n\geq 3$.

\begin{theorem}\label{thm:main_conf_Laplace}
\textbf{GFRC} is true for the conformal Laplacian in all dimensions $n\geq 3$.
\end{theorem}

Very recently, Lai and Zhang  \cite{Lai-Zhang} proved the  \emph{\textbf{GFRC}} for the Laplacian on surfaces using a different approach. Hence,  the \emph{\textbf{GFRC}} is now completely resolved for $P_2^g$ in all dimensions $n\geq 2$.

The general strategy for proving Theorem \ref{thm:main_conf_Laplace} is similar to that of \cite{Chen-Shi}. However, there are serious difficulties in higher dimensions.
\begin{itemize}
    \item We need to construct asymptotic coordinate system for all dimensions of order $\tau>\frac{n-2}{2}$.  The inversion of the usual geodesic normal coordinate is far from enough.
    \item Accordingly, to prove the vanishing of mass, we need much more computations even though following the same strategy. Some exceptional phenomenon occurs in dimensions $n=4k,k\geq 2$.
\end{itemize}

The key idea for solving these problems is to study the contact order of $M$ with its osculating sphere at $Q$,  with which we construct the  \emph{adapted} asymptotic coordinate system for all dimensions. Here for an umbilical point $Q$ of $M$ with principal curvature $\lambda$, the osculating sphere of $M$ at $Q$ is the sphere of radius $\frac{1}{|\lambda|}$ that tangent to $M$ at $Q$ with the same outer normal vector. To this end, we are able to prove a general theorem for hypersurfaces, which should be of independent interest.

\begin{theorem}\label{thm:main_ADM_mass}
    Let $M^n, n \geq 3$ be a closed embedded hypersurface of $\RR^{n+1}$ with induced metric $g$. Assume that $Q \in M$ is umbilical. Write $\rho:=\|\cdot-Q\|^2$, and set $\hat{g}:=\rho^{-2}g$. Then the following are true.
    \begin{enumerate}[(1)]
        \item If $n=3,4,5$, then $R_{\hat g}\in L^1(M,\hat g)$;
        \item if $n\geq 6$ and $R_{\hat g}\in L^1(M \backslash \{Q\},\hat g)$, or  $n=3,4,5$, then there exists an asymptotic coordinate system such that when $4 \nmid n$,
 $(M\setminus\{Q\},\hat g)$ is asymptotically flat of order $\lfloor \frac{n}{2} \rfloor+1$; when $4 \mid n$, $(M\setminus\{Q\},\hat g)$ is asymptotically flat of order $\frac{n}{2}$.
    \end{enumerate}
    Furthermore,  the ADM mass $m(\hat g)$ of $(M\setminus\{Q\},\hat g)$ is non-negative, and $m(\hat g)=0$ except for $n=4k, k \geq 2$.
\end{theorem}

\begin{remark}
The integrability of $R_{\hat g}$ ensures that the contact order of $M$ with its osculating sphere at $Q$ is at least  $\lfloor \frac{n}{2} \rfloor+1$ (resp. $\frac{n}{2}$ )  when $4 \nmid n$ (resp. $4 \mid n$). As we shall show, similar phenomena also occur for the $Q$-curvature.
\end{remark}
\begin{remark}
When $n=4k, k\geq 2$,  there are indeed examples whose mass is strictly positive under the same assumptions of Theorem \ref{thm:main_ADM_mass}; see Example \ref{example}.
\end{remark}

Thanks to Theorem \ref{thm:main_ADM_mass}, the same argument of \cite{Chen-Shi} applies to establish Theorem \ref{thm:main_conf_Laplace} when $4\nmid n$. However, to cope with the $4\mid n$ case,  a direct consequence of Example  \ref{example} is that the local analysis in the proof of Theorem \ref{thm:main_ADM_mass} fails.  At this point, we instead need a new mass formula for asymptotically flat hypersurfaces, which was once discovered by Lam in his thesis \cite{Lam} for entire graphs.

\begin{theorem}\label{thm:main_massFormula}
    Let $N$ be a complete non-compact hypersurface in $\RR^{n+1}$ with induced metric $g$. Assume that there is a compact subset $K\subset N$ such that $N\setminus K$ is the graph of a smooth function $\psi: \RR^n\setminus B_R(0)\to\RR$, satisfying
    $$|\partial_\alpha\psi|+|y||\partial^2_{\alpha\beta}\psi|+|y|^2|\partial^3_{\alpha\beta\gamma}\psi|=O(|y|^{-\frac{\tau}{2}}),$$
    with $\tau>
    \frac{n-2}{2}$. Then there is a global tangent vector field $X$ on $N$ such that $\div_g X=\langle e_{n+1},\bold n\rangle R_g$ where $\bold n$ is the unit normal vector field of $N$ such that $\langle e_{n+1},\bold n\rangle>0$ on the end $N\setminus K$ for $|y|\gg 1$, and that
    $$m(g)=\frac{1}{|\Sp^{n-1}|}\int_N \div_g X\ud V_g=\frac{1}{|\Sp^{n-1}|}\int_N \langle e_{n+1},\bold n\rangle R_g \ud V_g.$$
In particular, if $R_g\equiv 0$, then $m(g)=0$.
\end{theorem}

Since the metric $\hat g$ in the proof of Theorem \ref{thm:main_conf_Laplace} has zero scalar curvature, the above theorem in fact implies that the mass automatically vanishes. Hence, this concludes the proof of Theorem \ref{thm:main_conf_Laplace} using the same strategy as in \cite{Chen-Shi}.

\medskip

We proceed to study the \emph{\textbf{GFRC}} for  the Paneitz operator $P_4^g$. In parallel to Theorem \ref{thm:main_conf_Laplace}, the $Q$-curvature version of the positive energy theorem in \cite{GM,HY1,HY2,ALL} plays a key role. 

In the past decade, some great advances have been made in the constant $Q$-curvature problem  on a closed manifold $(M^n,g)$. For $n\geq 5$ M. Gursky and A. Malchiodi \cite{GM} proposed a sufficient condition of $\mathrm{ker}P_4^g=\{0\}$ that $R_g \geq 0$ and  $Q_g \gneqq 0$, which leads to the existence of  conformal metrics with positive constant $Q$-curvature. Subsequently, Hang and Yang \cite[Proposition 1.1]{HY3} improved the sufficient condition to the hypothesis that $Y(M,[g])>0$ and $Q_g \gneqq 0$. The most general existence result for $P_4^g$ and $n \geq 6$ was proved by Gursky, Hang and Lin \cite{GHL}. For $n=3$, Hang and Yang \cite{HY2} proved a PMT theorem associated to the Paneitz operator.

In a recent paper \cite{ALL}, on an asymptotically flat manifold $(N^n, g), n\geq 5$ of order $\tau>\frac{n-4}{2}$ and regularity order $\ell=4$, R. Avalos, P. Laurain and J. Lira  introduced a fourth-order mass (called ``energy" by the authors) by
\begin{align*}
\mathsf{m}_4(g)=\frac{1}{|\Sp^{n-1}|}\lim_{t \to \infty}\int_{\Sp_t^{n-1}}(\pa_\beta \pa_\gamma \pa_\gamma  g_{\alpha \alpha}-\pa_\beta \pa_\gamma \pa_\alpha g_{\alpha \gamma})\nu^\beta \ud \sigma_{\Sp_t^{n-1}},
\end{align*}
which had been defined earlier by B. Michel \cite{Michel} for higher-order GJMS operators. 
Furthermore, the authors \cite{ALL} established a  positive energy theorem; see Section \ref{Sect:Paneitz} for precise statement.

Next we establish a $Q$-curvature analogue to Theorem \ref{thm:main_ADM_mass}.

\begin{theorem}\label{thm:main_4th_mass}
    Let $M^n, n \geq 5$ be a closed embedded hypersurface of $\RR^{n+1}$ with induced metric $g$. Assume that $Q \in M$ is umbilical. Write $\rho:=\|\cdot-Q\|^2$ and set $\hat{g}:=\rho^{-2}g$. Then the following are true.
    \begin{enumerate}[(1)]
        \item If $n=5,6,7$, then $Q_{\hat g}\in L^1(M,\hat g)$, and $(M\setminus\{Q\},\hat g)$ is asymptotically flat of order $3$;
        \item if $n\geq 8$ and  $Q_{\hat g}\in L^1(M \backslash \{Q\},\hat g)$, there exists an asymptotic coordinate system such that when when $n \neq 4j+2, j\geq 2$,
 $(M\setminus\{Q\},\hat g)$ is asymptotically flat of order $\lfloor \frac{n}{2} \rfloor$; when $n=4j+2, j\geq 2$, $(M\setminus\{Q\},\hat g)$ is asymptotically flat of order $ \frac{n}{2}-1$.
    \end{enumerate}
    Furthermore,  the fourth-order energy $\mathsf{m}_4(\hat g)$ of $(M\setminus\{Q\},\hat g)$ is nonnegative, and $\mathsf{m}_4(\hat g)=0$ except for $n=4j+2, j \geq 2$.
\end{theorem}

Using the same strategy as the conformal Laplacian case, we make elaborate efforts to establish the following Green function rigidity for the Paneitz operator.

\begin{theorem}\label{thm:Paneitz_oper}
For $n\geq 5$ and $n\neq 4j+2, j\geq 2$, let $(M^n,g)$ be a closed embedded hypersurface in $\RR^{n+1}$ such that the Yamabe constant $Y(M,[g])>0$. Then \textbf{GFRC} is true for the Paneitz operator $P_4^g$. When $n=3$, under the same assumption, $(M^3, g)$ is globally conformal to $\Sp^3$.
\end{theorem}

At present,  we have to exclude the case  $n=4j+2, j\geq 2$ in the statement of Theorem \ref{thm:Paneitz_oper}, since the vanishing of $\mathsf{m}_4(\hat g)$ in such dimensions remains open.  As in Theorem \ref{thm:main_conf_Laplace}, a global analysis together with the fact $Q_{\hat g}=0$ should come into play.  It is interesting to know whether we can find a nice formula for $\mathsf{m}_4$ similar to that of Theorem \ref{thm:main_massFormula}.  Moreover, the $n=3,4$ cases differ significantly from other cases and need to be further explored.

In the next section, we describe the relationship between the integrability of $R_{\hat g}$ and the contact order of $M$ with its osculating sphere at an umbilical point $Q$. This is crucial for subsequent discussions. Then in Section \ref{Sect:Asym_coor_mass}, we first prove Theorem \ref{thm:main_ADM_mass} by introducing the \emph{adapted} asymptotic coordinates, and derive the mass formula in Theorem \ref{thm:main_massFormula} for asymptotically flat hypersurfaces those are graphs outside a compact subset. Combining ideas from \cite{Chen-Shi}, we prove Theorem \ref{thm:main_conf_Laplace}. Finally, similar ideas apply to prove Theorems \ref{thm:main_4th_mass},\ref{thm:Paneitz_oper} for the Paneitz operator, and such proofs occupy Sections \ref{Sect4}, \ref{Sect:Paneitz}. There are two appendices, those should be of independent interest. The first one is an application of our method of proving Theorem \ref{thm:main_massFormula}: we generalize a formula of Reilly for graphs to general Euclidean hypersurfaces. The second one is an elementary and simplified proof of the known fact (cf. \cite{ALL}) that the mass of Green function for the Paneitz operator in \cite{GM, HY1} is a  constant multiple of $\mathsf{m}_4(\hat g)$.

\section{Integrability of scalar curvature}

Let $M$ be an embedded hypersurface in $\RR^{n+1}$ with induced metric $g$, and $Q\in M$ be an umbilical point. Without loss of generality, we may assume $Q=0\in\RR^{n+1}$.  Define $\rho:=\|P-Q\|^2=\langle \mathbf{r},\mathbf{r}\rangle$, and $\eta:=\langle \mathbf{r},\mathbf{n}\rangle$, where $\mathbf{r}$ is the position vector of $M$ in $\RR^{n+1}$  and $\mathbf{n}$ is the inward unit normal vector field. Given any point $P\in M$, we choose an orthonormal frame of $M$, $\{e_\alpha; 1 \leq \alpha \leq n\}$. Then we have $\nabla_g\rho=2\sum_\alpha\langle \mathbf{r},e_\alpha\rangle e_\alpha=2(\mathbf{r}-\mathbf{r}^{\perp})=2\mathbf{r}-2\eta\mathbf{n}$ and consequently
$$|\nabla_g\rho|^2=\sum_\alpha 4\langle \mathbf{r},e_\alpha\rangle^2=4\|\mathbf{r}\|^2-4\langle \mathbf{r},\mathbf{n}\rangle^2=4\rho-4\eta^2.$$
Also, at $P$ we have
$$\nabla_g^2\rho(e_\alpha,e_\beta)=\langle\nabla_{e_\alpha}^g (\nabla_g\rho),e_\beta\rangle=2\delta_{\alpha\beta}+2\eta II(e_\alpha,e_\beta),$$
where $II(e_\alpha,e_\beta)=-\langle \nabla_{e_\alpha}\mathbf{n},e_\beta\rangle$ is the second fundamental form. As a consequence,
$$\Delta_g\rho=2n+2\eta H.$$
To sum up,  in any coordinate system we have
\begin{equation}\label{eqn:nabla_lap_rho}
      |\nabla_g \rho|^2=4\rho-4\eta^2,\quad \rho_{,\alpha\beta}=2g_{\alpha\beta}+2\eta h_{\alpha\beta},\quad \Delta_g\rho=2n+2\eta H.  
    \end{equation}

Locally $M$ is given as the graph of a smooth function $f:B_\delta(0)\to \RR$ with $f(0)=0$ and $\nabla f(0)=0$. Here, $\nabla$ means the gradient with respect to the flat metric on $\RR^n$. Then under the coordinate chart $x=(x_1,\cdots, x_n)$, the induced metric on $M$ is $g_{\alpha\beta}=\delta_{\alpha\beta}+f_\alpha f_\beta$, where $f_\alpha$ means $\frac{\partial f}{\partial x_\alpha}$. It is straightforward that $\rho=|x|^2+f^2(x)$ 
    and
    \begin{equation*}
    \eta=\frac{f-x\cdot\nabla f}{\sqrt{1+|\nabla f|^2}}.
    \end{equation*}
Also, it is direct to check that the second fundamental form is  $$II=\sum_{\alpha,\beta}h_{\alpha\beta}\ud x_\alpha\otimes \ud x_\beta=\frac{1}{\sqrt{1+|\nabla f|^2}}\sum_{\alpha,\beta}f_{\alpha\beta}\ud x_\alpha\otimes \ud x_\beta$$ and the mean curvature is $$H=\mathrm{tr}_g(II)=\frac{1}{\sqrt{1+|\nabla f|^2}}g^{\alpha \beta}f_{\alpha \beta}.$$
At $Q$, we have
$$g_{\alpha\beta}=\delta_{\alpha\beta},\quad\Gamma_{\alpha\beta}^\gamma=0,\quad h_{\alpha\beta}=f_{\alpha\beta}(0),\quad \mathrm{~~and~~}\  H=\Delta f(0),$$
where $\Delta$ denotes the Euclidean Laplacian. In the following, $(x_1,\cdots,x_n)$ will always be the above graph coordinate system.  

\medskip

It is well known that if the order $\tau$ of an asymptotically flat $n$-manifold is $>\frac{n-2}{2}$, and the scalar curvature is $L^1$, then the ADM mass is well defined and independent of the asymptotic coordinates (cf. \cite{bartnik}). Now we are ready to study the integrability of the scalar curvature. Since the integral of the scalar curvature is geometric and then coordinate free, we prefer to work under the coordinate system $x=(x_1,\cdots,x_n)$.

Under the conformal change of metric $\hat g:=\rho^{-2}g$, the scalar curvature transforms by
\begin{align}\label{eqn:scalar_curv}
   R_{\hat g}=&\rho^2\Big(R_g+2(n-1)\frac{\Delta_g \rho}{\rho}-n(n-1)\frac{|\nabla^g\rho|^2}{\rho^2}\Big) \no\\
   =& \rho^2\Big(R_g+4(n-1)H\frac{\eta}{\rho}+4n(n-1)\frac{\eta^2}{\rho^2}\Big),
\end{align}
where we used \eqref{eqn:nabla_lap_rho} in the last equality. Fix an umbilical point $Q\in M$. Then in a punctured neighborhood of $Q$, we have
\begin{equation}\label{eqn:RdV}
    R_{\hat g}\ud V_{\hat g}=\rho^{2-n}\Big(R_g+4(n-1)H\frac{\eta}{\rho}+4n(n-1)\frac{\eta^2}{\rho^2}\Big) \ud V_g. 
\end{equation}
Notice that $R_g=H^2-|II|^2$ by Gauss equation. Clearing a common nonzero factor $(1+|\nabla f|^2)^{-1}$ in \eqref{eqn:scalar_curv}, we only need to expand 
 $$\mathcal{R}:=4n(n-1)\Big(\frac{f-x\cdot \nabla f}{\rho}\Big)^2+4(n-1)g^{\alpha\beta}f_{\alpha\beta}\frac{f-x\cdot \nabla f}{\rho}+(g^{\alpha\beta}f_{\alpha\beta})^2-g^{\alpha\mu}g^{\beta\nu}f_{\alpha\beta}f_{\mu\nu}$$
 around $Q$. For brevity, let $B$ denote a matrix $(B^\alpha_\nu)$ with the element $B^\alpha_\nu:=g^{\alpha\beta}f_{\beta\nu}$, then the above equation becomes
 \begin{equation}\label{R:concise}
\mathcal{R}=4n(n-1)\Big(\frac{f-x\cdot\nabla f}{\rho}+\frac{1}{2n}\tr B\Big)^2-\tr\Big( \mathring{B}^2\Big),
\end{equation}
where $\mathring{B}^\alpha_\nu=B^\alpha_\nu-\frac{\tr B}{n}\delta^\alpha_\nu$ is the trace-free part of $B$.

\begin{proposition}\label{prop:scalar_curv_order}
    There hold $\mathcal{R}=O(|x|^2)$ and  $R_{\hat g}=O(|x|^6)$ near $Q$.
\end{proposition}

\begin{proof} From the assumption that $Q$ is umbilical,  near $Q$ we write
 $$f(x)=\frac{H(Q)}{2n}|x|^2+A_3(x)+O(|x|^4),$$
 where $A_3(x)$ denotes a homogeneous polynomial of degree $3$. Set $r:=|x|, A_3(x)=r^3 A_3(\theta)$ and $\Delta_\theta:=\Delta_{\Sp^{n-1}}$, where $\theta=x|x|^{-1}\in \Sp^{n-1}$. Short $H$ for $H(Q)$, and $A_3$ for $A_3(x)$. Then
 \begin{equation*}
     x\cdot\nabla f=\frac{H}{n}r^2+3A_3(\theta)r^3+O(r^4)
 \end{equation*}
and thereby
\begin{equation*}
    f-x\cdot\nabla f=-\frac{H}{2n}r^2-2A_3(\theta)r^3+O(r^4).
\end{equation*}
On the other hand, we have
\begin{equation*}
    \rho=|x|^2+f^2(x)=r^2\Big(1+\frac{H^2}{4n^2}r^2+O(r^3)\Big).
\end{equation*}
Putting the above two equations together yields
\begin{equation*}
\frac{f-x\cdot \nabla f}{\rho}=-\frac{H}{2n}-2A_3(\theta)r+O(r^2).    
\end{equation*}

Observe that
\begin{align*}    
B_\alpha^\nu=& g^{\alpha\mu}f_{\mu\nu}\\
=&\Big(\delta^{\alpha\mu}-\frac{f_\alpha f_\mu}{1+|\nabla f|^2}\Big)f_{\mu\nu}=\Big(\delta^{\alpha\mu}-f_\alpha f_\mu\Big)f_{\mu\nu}+O(r^3)\\
=&\Big(\delta^{\alpha\mu}-\frac{H^2}{n^2}x_\alpha x_\mu+O(r^3)\Big)\cdot\Big(\frac{H}{n}\delta_{\mu\nu}+\partial^2_{\mu\nu}A_3+O(r^2)\Big)+O(r^2)\\
=&\frac{H}{n}\delta^\alpha_\nu+\partial^2_{\alpha\nu}A_3+O(r^2),
\end{align*}
thereby
$$\mathring B_\alpha^\nu=\partial^2_{\alpha\nu}A_3-\frac{\Delta A_3}{n}\delta_\alpha^\nu$$
and
\begin{align*}
\tr B=g^{\alpha\beta}f_{\alpha\beta}=&H+\Delta A_3+O(r^2)\no\\
=&H+\Big(3(n+1)A_3(\theta)+\Delta_{\theta}A_3(\theta)\Big)r+O(r^2).
\end{align*}
This follows that
\begin{align*}   
\tr \mathring B^2=&|\nabla^2 A_3|^2 -\frac{(\Delta A_3)^2}{n},\no\\
\frac{f-x\cdot\nabla f}{\rho}+\frac{1}{2n}\tr B=&\Big(-\frac{1}{2}+\frac{3}{2n}\Big)\Delta_{\theta}A_3(\theta))r+O(r^2).
\end{align*}
 Hence, we obtain $\mathcal{R}=O(r^2)$ and $R_{\hat g}=O(r^6)$ by \eqref{R:concise} and \eqref{eqn:scalar_curv}. 
\end{proof}

\begin{lemma}\label{lem:R_integrability}
  Suppose we have $$R_g+4(n-1)H\frac{\eta}{\rho}+4n(n-1)\frac{\eta^2}{\rho^2}=c(\theta)|x|^k+o(|x|^k),$$ with $c(\theta)$ not identically zero. Then $R_{\hat g} \in L^1(\hat M, \hat g)$ if and only if $k>n-4$. 
\end{lemma}

\begin{proof}
   Since $\rho=|x|^2(1+o(1))$, if $k>n-4$, then we deduce from 
 $$\int_{B_\delta(0)}|x|^{2(2-n)+k} \ud x<\infty$$
 that $R_{\hat g}\in L^1(\hat M, \hat g)$.

 On the other hand, assume that $R_{\hat g}\in L^1(\hat M,\hat g)$ and $c(\theta_0)\neq 0$ for some $\theta_0$. Then we may choose a spherical neighborhood $U\subset \Sp^{n-1}$ of $\theta_0$, over which $|c(\theta)|$ has a positive lower bound $c_0>0$. Let $C_U$ be the corresponding cone in $\RR^n$ emanating from the origin. Then for small $\delta>0$, by \eqref{eqn:RdV} we have 
 $$\rho^{-n}|R_{\hat g}|\geq \frac{c_0}{2}|x|^{2(2-n)+k}, \quad \forall~ x\in B_\delta(0)\cap C_U.$$
 Consequently, we have $\int_{B_\delta(0)\cap C_U}|x|^{2(2-n)+k} \ud x<\infty$, and hence $2(2-n)+k>-n$. This means $k>n-4$. 
\end{proof}

 Let $\lambda \in \RR$ be the principal curvature of $M$ at $Q$, and define
\begin{align*}
\phi (x) =\frac{\lambda |x|^2}{1 + \sqrt{1 - \lambda^2 |x|^2}},
\quad x \in B_{\frac{1}{|\lambda|}(0)}.
\end{align*}
Then a direct calculation yields 
\begin{equation}\label{graph_fcn}
\frac{\lambda}{2}(|x|^2+\phi^2)=\phi, \quad 1+|\nabla \phi|^2=(1-\lambda^2 |x|^2)^{-1},\quad \phi-x \cdot \nabla \phi=-\frac{\phi}{\sqrt{1-\lambda^2 |x|^2}}.
\end{equation}
It is convenient to use $\phi(x)=\lambda^{-1}(1 - \sqrt{1 - \lambda^2 |x|^2})$ for $\lambda \neq 0$. Note that the graph of $\phi$ is precisely the osculating sphere of $M$ at $Q$.

\begin{proposition}\label{prop:vanishing_order_g}
If $R_{\hat{g}} \in L^1(\hat M,\hat{g})$, then near $Q$  we have
\begin{enumerate}[(i)]
\item when $4 \nmid n$, $f(x) = \phi (x) + O(|x|^{\lfloor \frac{n}{2} \rfloor + 1})$;
\item when $4 \mid n$, $f(x) = \phi(x) + \kappa |x|^{\frac{n}{2}} + O(|x|^{\frac{n}{2} + 1}),~\kappa \in \RR$.
\end{enumerate}
\end{proposition}

To simplify the computation, we use the following formula for $\mathcal{R}$:
\begin{equation}\label{concise_R}
\mathcal{R}=4n(n-1)\Big(\frac{f-x\cdot\nabla f}{\rho}+\frac{g^{\alpha\beta}f_{\alpha\beta}}{2n}\Big)^2+\frac{(g^{\alpha\beta}f_{\alpha\beta})^2}{n}-g^{\alpha\mu}g^{\beta\nu}f_{\alpha\beta}f_{\mu\nu}.
\end{equation}
As before, we denote by $B^\alpha_\nu:=g^{\alpha\beta}f_{\beta\nu}$ and its trace-free part $\mathring{B}^\alpha_\nu=B^\alpha_\nu-\frac{\tr B}{n}\delta^\alpha_\nu$. Then we have
$$\mathcal{R}=4n(n-1)\Big(\frac{f-x\cdot\nabla f}{\rho}+\frac{1}{2n}\tr B\Big)^2-\tr\Big( \mathring{B}^2\Big).$$

\begin{proof}[Proof of Proposition \ref{prop:vanishing_order_g}.]
Let $m = \lfloor \frac{n}{2} \rfloor$ be the integer part and $r=|x|$.  Recall that $f(0)=0, \nabla f(0)=0$. Suppose 
$$f(x)-\phi(x)=A_k(x)+O(r^{k+1})$$
for some $3 \leq k \leq m$, where $A_k$ is a homogeneous polynomial of degree $ k$, we shall prove by induction that $A_k=0$ except when $4|n$ and $k=m=\frac{n}{2}$. Short $A$ for $A_k$. Then we have 
$$
f(x) = \phi(x) + A(x) + O(r^{k+1})
$$
near $Q$.

Notice that
\begin{align*}
f - x \cdot \nabla f &= \phi  - x \cdot \nabla \phi + (A - x \cdot \nabla A ) + O(r^{k+1})  \\
&= -\frac{\phi}{\sqrt{1 - \lambda^2 r^2}} + (A - x \cdot \nabla A ) + O(r^{k+1})\\
&= -\frac{\phi}{\sqrt{1 - \lambda^2 r^2}} -(k-1)A + O(r^{k+1})
\end{align*}
and 
\begin{align*}
\rho = r^2 + f^2 &= r^2 + \phi^2 + 2\phi A + O(r^{k+3})\\
&= 2\lambda^{-1} \phi  +  2\phi A + O(r^{k+3}) \\
&= 2 \lambda^{-1} \phi \left(1 + \lambda A + O(r^{k+1})\right),
\end{align*}
thereby
$$
\rho^{-1} = \frac{\lambda}{2} \phi^{-1} \left(1 - \lambda A + O(r^{k+1})\right).
$$
Hence putting the above equations together yields
\begin{equation}\label{eta/rho}
\frac{f - x \cdot \nabla f}{\rho} = -\frac{\lambda}{2} \frac{1}{\sqrt{1 - \lambda^2 r^2}} -\frac{(k-1)\lambda}{2} \phi^{-1} A  + O(r^{k-1}).
\end{equation}

Observe that
\begin{align*}
1 + \vert \nabla f \vert^2 &= 1 + \vert \nabla \phi \vert^2 +O(r^{k}) = \frac{1}{1 - \lambda^2 r^2} \left( 1  + O(r^{k}) \right)
\end{align*}
and then
$$
(1 + \vert \nabla f \vert^2)^{-1} = (1 - \lambda^2 r^2)\left( 1 + O(r^{k}) \right).
$$
Now we have
\begin{align*}
g^{\alpha \beta} =& \delta^{\alpha \beta} - \frac{f_\alpha f_\beta}{1 + \vert \nabla f \vert^2} \\
=& \delta^{\alpha \beta} - \left( \frac{\lambda x_\alpha}{\sqrt{1 - \lambda^2 r^2}} + O(r^{k-1})\right)\left( \frac{\lambda x_\beta}{\sqrt{1 - \lambda^2 r^2}}  + O(r^{k-1})\right) \cdot (1 - \lambda^2 r^2)\left( 1 + O(r^{k}) \right) \\
=& \delta^{\alpha \beta} - \lambda^2 x_\alpha x_\beta   + O(r^{k}).
\end{align*}
This gives
\begin{align*}
B^\alpha_\nu=g^{\alpha \beta} f_{\beta \nu} 
=& \left( \delta^{\alpha \beta} - \lambda^2 x_\alpha x_\beta +O(r^k)\right) \cdot \left(\frac{\lambda \delta_{\beta \nu}}{\sqrt{1- \lambda^2 r^2}} + \frac{\lambda^3 x_\beta x_\nu}{(1- \lambda^2 r^2)^{\frac{3}{2}}} + A_{\beta \nu} +O(r^{k-1})\right) \\
=& \frac{\lambda \delta_{\alpha \nu}}{\sqrt{1- \lambda^2 r^2}} + \frac{\lambda^3 x_\alpha x_\nu}{(1- \lambda^2 r^2)^{\frac{3}{2}}} + A_{\alpha \nu}  - \frac{\lambda^3 x_\alpha x_\nu }{\sqrt{1- \lambda^2 r^2}} - \frac{\lambda^5 r^2 x_\alpha x_\nu}{(1- \lambda^2 r^2)^{\frac{3}{2}}} +O(r^{k-1})\\
={}& \frac{\lambda \delta_{\alpha \nu}}{\sqrt{1 - \lambda^2 r^2}} + A_{\alpha \nu}+ O(r^{k-1}).
\end{align*}
Hence,
\begin{align}\label{H}
\tr B=g^{\alpha \beta} f_{\alpha \beta}  ={}& \frac{n \lambda}{\sqrt{1 - \lambda^2 r^2}} + \Delta A  + O(r^{k-1})
\end{align}
and then
\begin{align}\label{trace-freeB}
\mathring{B}^\alpha_\nu =& A_{\alpha\nu}-\frac{\Delta A}{n}\delta_{\alpha\nu} + O(r^{k-1}).   
\end{align}
Notice that $\frac{\lambda}{2} \phi^{-1} = r^{-2} + O(1)$. Combining (\ref{eta/rho}) and (\ref{H}) yields
\begin{align}\label{H*eta/rho}
4n(n-1)\Big(\frac{f - x \cdot \nabla f}{\rho}+\frac{\tr B}{2n}\Big)^2&=4n(n-1)\Big(\frac{\Delta A}{2n}-\frac{(k-1)A}{r^2}+O(r^{k-1})\Big)^2\no\\ 
&= 4n(n-1)\Big(\frac{\Delta A}{2n}-\frac{(k-1)A}{r^2}\Big)^2+O(r^{2k-3}).
\end{align}
Moreover, by (\ref{trace-freeB}) we have
\begin{align}\label{trace-freeB^2}
\tr\Big(\mathring{B}^2\Big) 
&= |\nabla^2 A|^2-\frac{(\Delta A)^2}{n}+ O(r^{2k-3}).
\end{align} 
Therefore, combining (\ref{H*eta/rho}) and (\ref{trace-freeB^2}) we obtain
\begin{equation}\label{order_R:induction}
\mathcal{R} = 4n(n-1) (k-1)^2\frac{A^2}{r^4} -4(n-1)(k-1) \frac{A}{r^2} \Delta A + (\Delta A)^2 - \vert \nabla^2 A \vert^2 + O(r^{2k-3}).
\end{equation}
Since $\mathcal{R} = O(r^{n-3})$, when $k \leq m$ we have
\begin{equation}\label{A_k}
4n(n-1)(k-1)^2 A^2 - 4(n-1)(k-1) r^2 A \Delta A + r^4 \left[ (\Delta A)^2 - \vert \nabla^2 A \vert^2 \right] = 0.
\end{equation}

\begin{lemma}
    If $A=A_k$ satisfies \eqref{A_k} for $2k<n$, then $A\equiv 0$.
\end{lemma}

\begin{proof}
    In terms of polar coordinates, \eqref{A_k} is equivalent to
    \begin{align*}
        (n-1)(k-2)^2(n-2k)A^2(\theta)-2(k-2)(n-k-1)A(\theta)\Delta_\theta A(\theta)\\
        -2(k-1)^2|\nabla_\theta A(\theta)|^2+(\Delta_\theta A(\theta))^2-|\nabla^2_\theta A(\theta)|^2=0.
    \end{align*}
Integrating over $\Sp^{n-1}$ and using the Bochner identity on spheres, we obtain
$$(n-2k)\int_{\Sp^{n-1}}(n-1)(k-2)^2A^2(\theta) +(2k-3)|\nabla_\theta A(\theta)|^2 \ud\sigma=0,$$
from which we conclude that $A(\theta)\equiv 0$.
\end{proof}

When $2k=n$, we obtain
\begin{equation}
   -2(k-2)(k-1)A(\theta)\Delta_\theta A(\theta)\\
        -2(k-1)^2|\nabla_\theta A(\theta)|^2+(\Delta_\theta A(\theta))^2-|\nabla^2_\theta A(\theta)|^2=0.
\end{equation}

This in turn implies $r^2\mid A_k^2$. Furthermore, since $r^2$ is an irreducible polynomial, we obtain $r^2 \mid A_k$.

\begin{claim}
 If $r^{2l} \mid A_k$, then $r^{2l+2} \mid A_k$, where $2l \leq k$ and $4l \neq n$.
\end{claim}

To this end, we write $A_k (x) = r^{2l} p(x)$, where $p(x)$ is a homogeneous polynomial of degree $k-2l$. Then
\begin{align*}
\pa_{\alpha \beta}^2 A_k
=& 2l  r^{2l-2} p(x)\delta_{\alpha \beta}  + 4l(l-1) r^{2l-4}x_\alpha x_\beta p(x) \\
&+ 2l  r^{2l-2} (x_\alpha \pa_\beta p(x) + x_\beta \pa_\alpha p(x)) + r^{2l} \pa_{\alpha \beta}^2 p(x).
\end{align*}
So,
\begin{align*}
\frac{1}{r^{2l-2}} \Delta A_k =& 2nlp(x) + 4l(l-1) p(x) + 4l(k-2l)p(x) + r^2 \Delta p(x) \\
\equiv& 2l(n+2k-2l-2) p(x) \qquad \mod r^2
\end{align*}
and 
\begin{align*}
\frac{1}{r^{4l-4}} \vert \nabla^2 A_k \vert^2 
\equiv& 4nl^2 p^2(x) + 16 l^2(l-1)p^2(x) + 16l^2 p(x) x_\alpha \pa_\alpha p(x) \\
& + 16l^2(l-1)^2 p^2(x) + 32l^2 (l-1) p(x)x_\alpha \pa_\alpha p(x) \\
& + 8l(l-1) p(x) x_\alpha x_\beta \pa_{\alpha \beta}^2 p(x) + 8l^2 x_{\alpha} x_\beta \pa_\alpha p(x) \pa_\beta p(x) \qquad \mod r^2 \\
\equiv& \left[4l^2 (n+2k^2 - 4k -4l^2 + 4l) + 8l(l-1)(k-2l)(k-2l-1) \right] p^2(x) \mod r^2.
\end{align*}
Inserting these terms into (\ref{A_k}) gives
\begin{align}\label{cal:coeff_R}
0 =& 4n(n-1)(k-1)^2 \frac{A_k^2}{r^{4l}} - 4(n-1)(k-1)\frac{A_k}{r^{2l}} \frac{\Delta A_k}{r^{2l-2}} + \frac{1}{r^{4l-4}} \left[ (\Delta A_k)^2 - \vert \nabla^2 A_k \vert^2 \right] \no\\
\equiv& \big[ 4n(n-1)(k-1)^2 - 4(n-1)(k-1)\cdot 2l(n+2k-2l-2) + 4l^2 (n+2k-2l-2) \no\\
& -  4l^2 (n+2k^2 - 4k -4l^2 + 4l) - 8l(l-1)(k-2l)(k-2l-1)\big] p^2 (x) \quad \mod r^2 \no\\
\equiv& 4(k-l-1) \left[ (k-l-1) (n-4l)(n-1) + 2l(k-2l)\right] p^2 (x) \qquad \mod r^2.
\end{align}
Hence, when  $2l \leq k$ and $4l \neq n$, the above coefficient does not vanish, so $r^2 \mid p^2(x)$, thereby $r^2 \mid p(x)$. This directly yields $r^{2l+2} \mid A_k $. We finish the proof of the claim.

\medskip

When $k \leq m-1$, the above claim implies $r^{2(\lfloor \frac{k}{2} \rfloor + 1)} \mid A_k$, but $\deg A_k = k < 2(\lfloor \frac{k}{2} \rfloor + 1)$. This forces
$$
A_k(x) = 0, \quad k\leq m-1.
$$
Finally, our discussion is divided into two cases.
\begin{enumerate}[(i)]
\item If $4 \nmid n$, then $r^{2(\lfloor \frac{m}{2} \rfloor + 1)} \mid A_m$, also $\deg A_m = m < 2(\lfloor \frac{m}{2} \rfloor + 1) $, so $A_m (x) = 0$.
\item If $4 \mid n$, then $r^m \vert A_m$, so $A_m(x) = \kappa r^m$ for some constant $\kappa$.
\end{enumerate}
This completes the proof.
\end{proof}

\section{Adapted asymptotic coordinates and the ADM mass}\label{Sect:Asym_coor_mass}

Throughout this section,  we let $\lambda=\frac{H(Q)}{n}$ and fix $m = \lfloor \frac{n}{2} \rfloor$ as before.

 The following elementary lemma demonstrates that for all $n \geq 3$, there exists an \emph{adapted} asymptotic coordinate system for  $(M \backslash \{Q\},\hat g)$ such that the mass $m(\hat g)$ is well defined thanks to Proposition \ref{prop:vanishing_order_g}. 
\begin{lemma}\label{lem:asym_order-coor}
If $f(x) = \phi(x) + O(|x|^k)$ for some integer $k \geq 3$, then there is an adapted asymptotic coordinate system  
\begin{equation}\label{coor:asymp_natural}
y = \frac{\lambda}{2}\frac{x}{\phi(x)}=\frac{1+\sqrt{1-\lambda^2 |x|^2}}{2} \frac{x}{|x|^2}
\end{equation}
such that $(M\backslash \{Q\}, \hat g)$ is asymptotically flat of order $k$.
\end{lemma}
\begin{proof}

Let $r=|x|$ and $t=|y|$.
A direct computation together with \eqref{graph_fcn} yields
$$
\frac{\lambda^2}{4} + \vert y \vert^2 = \frac{\lambda^2}{4} + \frac{\lambda^2}{4}\frac{\vert x \vert^2}{\phi^2} = \frac{\lambda}{2} \frac{1}{\phi(x)}.
$$
This gives
$$\phi(x) = \frac{\lambda}{2}\frac{1}{\frac{\lambda^2}{4} + \vert y \vert^2}$$
and the inverse of the above asymptotic coordinate map is
\begin{equation}\label{inverse_coor:asymp_natural}
x = \frac{y}{\frac{\lambda^2}{4} + \vert y \vert^2}.
\end{equation}
In terms of the variable $y$, we can express 
$$
f(y) = \frac{\lambda}{2} \frac{1}{\frac{\lambda^2}{4} + \vert y \vert^2} + \sum_{i=k}^{2k-3} \widetilde{A}_i (y) + O(\vert y \vert^{2-2k}),
$$
where $\widetilde{A}_i (y)$ is a homogeneous polynomial  of degree $-i$. 

\begin{remark}\label{rem:derivative_decay}
Since $f$ is  smooth by assumption and the coordinate change is explicit, it is evident that the $O(\vert y \vert^{2-2k})$ term behaves nice when taking derivatives, and is in fact $O_l(\vert y \vert^{2-2k})$ for any $l$ by a direct mathematical induction argument. This fact will be frequently used later.
\end{remark}

Notice that
\begin{align*}
\rho &= r^2 + f^2 \\
&= \frac{t^2}{(t^2 + \frac{\lambda^2}{4})^2} + \frac{\lambda^2}{4} \frac{1}{(t^2 + \frac{\lambda^2}{4})^2}+ \frac{\lambda}{t^2 + \frac{\lambda^2}{4}} \sum_{i=k}^{2k-3} \widetilde{A}_i + O(t^{-2k}) \\
&= \frac{1}{t^2 + \frac{\lambda^2}{4}}\big(1+ \lambda \sum_{i=k}^{2k-3} \widetilde{A}_i + O(t^{2-2k})\big).
\end{align*}
So,
$$
\rho^{-2} = (t^2 + \frac{\lambda^2}{4})^2 \big(1 - 2\lambda \sum_{i=k}^{2k-3} \widetilde{A}_i + O(t^{2-2k})\big).
$$

In terms of the new variable $y$ we obtain
\begin{align*}
g=&\ud x_\alpha \otimes \ud x_\alpha + \ud f \otimes \ud f \\
=& \ud\left( \frac{y_\alpha}{t^2 +\frac{\lambda^2}{4}} \right)\otimes \ud\left( \frac{y_\alpha}{t^2 +\frac{\lambda^2}{4}} \right)\\
&+ \ud\left(\frac{\lambda}{2} \frac{1}{t^2 + \frac{\lambda^2}{4}} + \sum_{i=k}^{2k-3} \widetilde{A}_i (y) + O(t^{2-2k})\right) \otimes \ud\left(\frac{\lambda}{2} \frac{1}{t^2 + \frac{\lambda^2}{4}} + \sum_{i=k}^{2k-3} \widetilde{A}_i (y) + O(t^{2-2k})\right) \\
=& \left( \frac{\ud y_\alpha}{t^2 + \frac{\lambda^2}{4}} - \frac{2y_\alpha t\ud t}{(t^2 + \frac{\lambda^2}{4})^2} \right) \otimes \left( \frac{\ud y_\alpha}{t^2 + \frac{\lambda^2}{4}} - \frac{2y_\alpha t\ud t}{(t^2 + \frac{\lambda^2}{4})^2} \right) \\
& + \frac{\lambda^2}{4} \frac{4t^2 \ud t\otimes \ud t}{(t^2 + \frac{\lambda^2}{4})^4} - \frac{\lambda t}{(t^2 + \frac{\lambda^2}{4})^2} \ud t \otimes \ud \big(\sum_{i=k}^{2k-3} \widetilde{A}_i \big) \\
&- \frac{\lambda t}{(t^2 + \frac{\lambda^2}{4})^2} \ud \big(\sum_{i=k}^{2k-3} \widetilde{A}_i \big) \otimes \ud t + O(t^{-2k-2}) \\
=& (t^2 + \frac{\lambda^2}{4})^{-2} \bigg( \ud y_\alpha \otimes \ud y_\alpha - \lambda t\ud t \otimes \ud \big(\sum_{i=k}^{2k-3} \widetilde{A}_i \big) - \lambda t \ud \big(\sum_{i=k}^{2k-3} \widetilde{A}_i \big) \otimes \ud t + O(t^{2-2k}) \bigg).
\end{align*}
Then we have
\begin{align*}
\hat{g} =& \rho^{-2} (\ud x_\alpha \otimes \ud x_\alpha + \ud f \otimes \ud f)\\
=& \big(1 - 2\lambda \sum_{i=k}^{2k-3} \widetilde{A}_i + O(t^{2-2k})\big)\\
& \cdot\bigg( \ud y_\alpha \otimes \ud y_\alpha - \lambda t\ud t \otimes \ud \big(\sum_{i=k}^{2k-3} \widetilde{A}_i \big) - \lambda t \ud \big(\sum_{i=k}^{2k-3} \widetilde{A}_i \big) \otimes \ud t + O(t^{2-2k}) \bigg) \\
=& \bigg(1 - 2\lambda \sum_{i=k}^{2k-3} \widetilde{A}_i \bigg) \ud y_\alpha \otimes \ud y_\alpha - \lambda t\ud t \otimes \ud \big(\sum_{i=k}^{2k-3} \widetilde{A}_i \big) \\
&- \lambda t \ud \big(\sum_{i=k}^{2k-3} \widetilde{A}_i \big) \otimes \ud t + O(t^{2-2k}).
\end{align*}
This directly yields
$$
\hat{g}_{\alpha\beta} = \delta_{\alpha \beta} + \hat{h}_{\alpha \beta}
$$
with  $\hat{h}_{\alpha \beta} = O(t^{-k})$, $\pa \hat{h}_{\alpha \beta} = O(t^{-k-1})$ and $\pa^2 \hat{h}_{\alpha \beta} = O(t^{-k-2})$. Hence $y$ is an asymptotically flat coordinate system of order $k$.
\end{proof}

Now we are ready to prove Theorem \ref{thm:main_ADM_mass}.

\begin{proof}[Proof of Theorem \ref{thm:main_ADM_mass}.]
    
 By Proposition \ref{prop:vanishing_order_g}, we have $f(x) = \phi(x) + O(r^{m+1})$  when $4 \nmid n$.  By Lemma \ref{lem:asym_order-coor}, we know that $y$ is an asymptotic coordinate system of order $m+1$.

Since $m + 1 = \lfloor \frac{n}{2} \rfloor + 1 > \frac{n-2}{2}$, the ADM mass is well defined. Observe that
\begin{align*}
\sum_\alpha \hat{g}_{\alpha\alpha} =& n  \bigg(1 - 2\lambda \sum_{i=m+1}^{2m-1} \widetilde{A}_i \bigg) + 2\lambda  \sum_{i=m+1}^{2m-1} i\widetilde{A}_i + O(t^{-2m}),\\
\hat{g}_{tt} =& \sum_{\alpha, \beta} \hat{g}_{\alpha\beta}y_\alpha y_\beta t^{-2} \\
=&  \bigg(1 - 2\lambda \sum_{i=m+1}^{2m-1} \widetilde{A}_i \bigg) + 2\lambda  \sum_{i=m+1}^{2m-1} i\widetilde{A}_i + O(t^{-2m}).
\end{align*}
Subsequently, in view of Remark \ref{rem:derivative_decay},
\begin{align*}
\pa_t (\hat{g}_{tt} - \sum_\alpha \hat{g}_{\alpha\alpha}) &= -(n-1) \pa_t\bigg(1 - 2\lambda \sum_{i=m+1}^{2m-1} \widetilde{A}_i \bigg)+ O(t^{-2m-1}) \\
&= -2(n-1)\lambda t^{-1}\sum_{i=m+1}^{2m-1} i\widetilde{A}_i + O(t^{-2m-1}), \\
\frac{1}{t}(n\hat{g}_{tt} - \sum_\alpha \hat{g}_{\alpha\alpha}) &= 2(n-1)\lambda t^{-1}\sum_{i=m+1}^{2m-1} i\widetilde{A}_i + O(t^{-2m-1}).
\end{align*}
Hence we obtain
\begin{align*}
m(\hat{g}) =& \frac{1}{|\Sp^{n-1}|}\lim_{t \to \infty} \int_{\vert y \vert = t}  \pa_t (\hat{g}_{tt} - \sum_\alpha \hat{g}_{\alpha\alpha}) + \frac{1}{t}(n\hat{g}_{tt} - \sum_\alpha \hat{g}_{\alpha\alpha})  \ud \sigma \\
=&\frac{1}{|\Sp^{n-1}|} \lim_{t \to \infty} \int_{\vert y \vert = t} O(t^{-2m-1}) \ud\sigma \\
=& \frac{1}{|\Sp^{n-1}|}\lim_{t \to \infty} \int_{\mathbb{S}^{n-1}} O(t^{n-2m-2}) \ud \sigma =0.
\end{align*}
For the $4\mid n$ case, we have the following lemma, which concludes the proof.
\end{proof}

\begin{lemma}
When $n = 4k, k \geq 2$, there hold 
$$
f(x) = \phi (x) + \kappa |x|^{\frac{n}{2}}+O(|x|^{\frac{n}{2}+1}), \qquad \kappa \in \RR,
$$
near $Q$, and under the asymptotic coordinates $y=\frac{\lambda}{2}\phi^{-1}(x)x$, $(M \backslash \{Q\},\hat g)$ is asymptotically flat of order $\frac{n}{2}$. Moreover, the mass of $\hat g$ is $m(\hat g)=\frac{(n-1)(n-4)^2}{4} \kappa^2$.
\end{lemma}
\begin{proof}
In terms of the variable $y$ we have
\begin{align*}
f(y) &= \frac{\lambda}{2}\frac{1}{t^2 + \frac{\lambda^2}{4}} + \frac{\kappa t^{\frac{n}{2}}}{(t^2 + \frac{\lambda^2}{4})^{\frac{n}{2}}}  \\
&= \frac{\lambda}{2}\frac{1}{t^2 + \frac{\lambda^2}{4}} + \kappa t^{-\frac{n}{2}} + \sum_{i=\frac{n}{2}+1}^{\infty} a_i t^{-i}, \quad a_i \in \RR.
\end{align*}
So,
\begin{align*}
\rho &= r^2 + f^2 \\
&= \frac{t^2}{(t^2 + \frac{\lambda^2}{4})^2} + \frac{\lambda^2}{4}\frac{1}{(t^2 + \frac{\lambda^2}{4})^2} + \frac{\kappa \lambda t^{-\frac{n}{2}}}{t^2 + \frac{\lambda^2}{4}} + \frac{\lambda}{t^2 + \frac{\lambda^2}{4}} \sum_{i=\frac{n}{2}+1}^{n-2} a_i t^{-i} + \kappa^2 t^{-n} + O(t^{-n-1}) \\
&= \frac{1}{t^2 + \frac{\lambda^2}{4}} \big[ 1 + \kappa \lambda t^{-\frac{n}{2}} + \lambda \sum_{i = \frac{n}{2}+1}^{n-2} a_i t^{-i} + \kappa^2 t^{2-n} + O(t^{1-n})  \big]
\end{align*}
and then
$$
\rho^{-2} = (t^2 + \frac{\lambda^2}{4})^2 \big( 1 - 2 \kappa \lambda t^{-\frac{n}{2}} -2 \lambda \sum_{i = \frac{n}{2}+1}^{n-2} a_i t^{-i} -2 \kappa^2 t^{2-n} + O(t^{1-n})  \big).
$$

Similarly as before, we have
\begin{align*}
g=&\ud x_\alpha \otimes \ud x_\alpha + \ud f \otimes \ud f \\
=& (t^2 + \frac{\lambda^2}{4})^{-2} \ud y_\alpha \otimes \ud y_\alpha - \frac{\lambda t}{(t^2 + \frac{\lambda^2}{4})^2}\ud t \otimes \ud \big( \kappa t^{-\frac{n}{2}} + \sum_{i=\frac{n}{2}+1}^{\infty} a_i t^{-i} \big) \\
&- \frac{\lambda t}{(t^2 + \frac{\lambda^2}{4})^2} \ud \big( \kappa t^{-\frac{n}{2}} + \sum_{i=\frac{n}{2}+1}^{\infty} a_i t^{-i} \big) \otimes \ud t \\
& + \ud \big( \kappa t^{-\frac{n}{2}} + \sum_{i=\frac{n}{2}+1}^{\infty} a_i t^{-i} \big) \otimes \ud \big( \kappa t^{-\frac{n}{2}} + \sum_{i=\frac{n}{2}+1}^{\infty} a_i t^{-i} \big) +O(t^{-n-3}) \\
=& (t^2 + \frac{\lambda^2}{4})^{-2} \ud y_\alpha \otimes \ud y_\alpha - \frac{\lambda t}{(t^2 + \frac{\lambda^2}{4})^2}\ud t \otimes \ud \big( \kappa t^{-\frac{n}{2}} + \sum_{i=\frac{n}{2}+1}^{n-2} a_i t^{-i} \big) \\
&- \frac{\lambda t}{(t^2 + \frac{\lambda^2}{4})^2} \ud \big( \kappa t^{-\frac{n}{2}} + \sum_{i=\frac{n}{2}+1}^{n-2} a_i t^{-i} \big) \otimes \ud t + \frac{n^2}{4} \kappa^2  t^{-n-2} \ud t \otimes \ud t + O(t^{-n-3}).
\end{align*}
Then,
\begin{align*}
\hat{g} =& \rho^{-2}(\ud x_\alpha \otimes \ud x_\alpha + \ud f \otimes \ud f) \\
=& \big( 1 - 2 \kappa \lambda t^{-\frac{n}{2}} -2 \lambda \sum_{i = \frac{n}{2}+1}^{n-2} a_i t^{-i} -2 \kappa^2 t^{2-n}  \big) \ud y_\alpha \otimes \ud y_\alpha \\
&- \lambda t \ud t \otimes  \ud \big( \kappa t^{-\frac{n}{2}} + \sum_{i=\frac{n}{2}+1}^{n-2} a_i t^{-i} \big) - \lambda t  \ud \big( \kappa t^{-\frac{n}{2}} + \sum_{i=\frac{n}{2}+1}^{n-2} a_i t^{-i} \big) \otimes \ud t \\
&+ \frac{n^2}{4} \kappa^2  t^{2-n} \ud t \otimes \ud t + O(t^{1-n}),
\end{align*}
equivalently,
\begin{align*}
\hat{g}_{\alpha\beta} =& \delta_{\alpha\beta} \big( 1 - 2 \kappa \lambda t^{-\frac{n}{2}} -2 \lambda \sum_{i = \frac{n}{2}+1}^{n-2} a_i t^{-i} -2 \kappa^2 t^{2-n}  \big) \\
&-\lambda (y_\alpha \pa_\beta + y_\beta \pa_\alpha)\big( \kappa t^{-\frac{n}{2}} + \sum_{i=\frac{n}{2}+1}^{n-2} a_i t^{-i} \big)  +  \frac{n^2}{4} \kappa^2  t^{-n} y_\alpha y_\beta + O(t^{1-n}).
\end{align*}
This indicates that $y$ is an asymptotically flat coordinate system of order $\frac{n}{2}$. So the mass is well defined.

Observe that
\begin{align*}
\sum_{\alpha} \hat{g}_{\alpha \alpha} =& n\big( 1 - 2 \kappa \lambda t^{-\frac{n}{2}} -2 \lambda \sum_{i = \frac{n}{2}+1}^{n-2} a_i t^{-i} -2 \kappa^2 t^{2-n}  \big) \\
& - 2 \lambda t\pa_t\big( \kappa t^{-\frac{n}{2}} + \sum_{i=\frac{n}{2}+1}^{n-2} a_i t^{-i} \big) +  \frac{n^2}{4} \kappa^2  t^{2-n}  + O(t^{1-n}) 
\end{align*}
and
\begin{align*}
\hat{g}_{tt} =& \big( 1 - 2 \kappa \lambda t^{-\frac{n}{2}} -2 \lambda \sum_{i = \frac{n}{2}+1}^{n-2} a_i t^{-i} -2 \kappa^2 t^{2-n}  \big) \\
& - 2 \lambda t\pa_t\big( \kappa t^{-\frac{n}{2}} + \sum_{i=\frac{n}{2}+1}^{n-2} a_i t^{-i} \big) + \frac{n^2}{4} \kappa^2  t^{2-n}  + O(t^{1-n}).
\end{align*}
Combining these terms above yields, as before,
\begin{align*}
\pa_t (\hat{g}_{tt} - \sum_\alpha \hat{g}_{\alpha\alpha}) =& -(n-1) \pa_t \big(1 - 2 \kappa \lambda t^{-\frac{n}{2}} -2 \lambda \sum_{i = \frac{n}{2}+1}^{n-2} a_i t^{-i} -2 \kappa^2 t^{2-n} + O(t^{1-n})\big) \\
=& 2\lambda(n-1) \pa_t (\kappa t^{-\frac{n}{2}} + \sum_{i = \frac{n}{2}+1}^{n-2} a_i t^{-i}) + 2(n-1)(2-n)\kappa^2 t^{1-n} + O(t^{-n}), \\
\frac{1}{t}(n\hat{g}_{tt} - \sum_\alpha \hat{g}_{\alpha\alpha}) &= -2 \lambda (n-1) \pa_t\big( \kappa t^{-\frac{n}{2}} + \sum_{i=\frac{n}{2}+1}^{n-2} a_i t^{-i} \big) + (n-1) \frac{n^2}{4} \kappa^2  t^{1-n}  + O(t^{-n}).
\end{align*}
This leads to
\begin{align*}
&\pa_t (\hat{g}_{tt} - \sum_\alpha \hat{g}_{\alpha\alpha}) + \frac{1}{t}(n\hat{g}_{tt} - \sum_\alpha \hat{g}_{\alpha\alpha})  \\
=& 2(n-1)(2-n)\kappa^2 t^{1-n} + (n-1)\frac{n^2}{4} \kappa^2  t^{1-n}  + O(t^{-n})  \\
=& \frac{(n-1)(n-4)^2}{4} \kappa^2 t^{1-n} + O(t^{-n}).
\end{align*}
Consequently, we obtain
\begin{align*}
m(\hat{g}) =& \frac{1}{|\Sp^{n-1}|}\lim_{t \to \infty} \int_{\vert y \vert = t}  \pa_t (\hat{g}_{tt} - \sum_\alpha \hat{g}_{\alpha\alpha}) + \frac{1}{t}(n\hat{g}_{tt} - \sum_\alpha \hat{g}_{\alpha\alpha})  \ud \sigma \\
=& \frac{1}{|\Sp^{n-1}|} \lim_{t \to \infty} \int_{\mathbb{S}^{n-1}} \frac{(n-1)(n-4)^2}{4} \kappa^2+O(t^{-1}) \ud \sigma \\
=& \frac{(n-1)(n-4)^2}{4} \kappa^2.
\end{align*}
\end{proof}

\begin{example}\label{example}
When $n=4k, k \geq 2$,  the vanishing order of $f-\phi$ in Proposition \ref{prop:vanishing_order_g} cannot be improved. For example,  locally we may take
$$f(x)=\phi(x)+\kappa |x|^{2k}+|x|^{2k} l(x)+O(|x|^{2k+2}),$$ 
where $\kappa \in \RR\setminus \{0\}$ and $l(x)$ is any linear function, and  extend this local graph to a closed hypersurface smoothly. Then we always have $R_{\hat g}=O(r^{n+2})$, that is,
$$\mathcal{R} = O(r^{n-2})\quad \Rightarrow \quad  m(\hat g)=\frac{(n-1)(n-4)^2}{4} \kappa^2>0.$$
Also, even if we impose a stronger condition that $R_{\hat{g}}=0$ in a punctured neighborhood of the umbilical point $Q$ and $f$ is real analytic,  it fails to obtain the vanishing of $\kappa$.
Since the above computations are similar to those of Proposition \ref{prop:vanishing_order_g}, we leave the details to interested readers.  
\end{example}

The above example demonstrates that in dimensions $n=4k+2, k\geq 2$, a global analysis is crucial to the resolution of \emph{\textbf{GFRC}},  which occupies the whole Section \ref{Sect4}.

\section{A new mass formula and GFRC for $P_2^g$}\label{Sect4}

Recall that locally $M$ is a graph of a smooth function $f:B_\delta(0)\to\RR$ near $Q$. Without loss of generality, we assume that $Q$ is the origin of $\RR^{n+1}$. Under the inversion $I_Q$ as in \eqref{def:inversion}, the inverted hypersurface becomes
$$\left\{\Big(\frac{x}{|x|^2+f(x)^2},\frac{f(x)}{|x|^2+f(x)^2}\Big);~\ x\in B_\delta(0)\subset\RR^n\right\}.$$
Now we introduce a new coordinate change by
\begin{equation}\label{eqn:AsymCord_new}
   y=\frac{x}{|x|^2+f(x)^2}.
\end{equation}
When $|x|\ll 1$, this gives an asymptotic coordinate system for the end of $I_Q(M\backslash\{Q\})$. Combining the fact that $f(x)=\phi(x)+O(|x|^{\frac{n}{2}})$ by Proposition \ref{prop:vanishing_order_g} and \eqref{graph_fcn} we obtain
$$\frac{\lambda}{2}(|x|^2+f^2(x))=\frac{\lambda}{2}(|x|^2+\phi^2(x))+O(|x|^{\frac{n}{2}+2})=\phi(x)+O(|x|^{\frac{n}{2}+2}).$$
So this new coordinate change is very close to that of Section \ref{Sect:Asym_coor_mass}. Using this coordinate chart $y$, the end of $I_Q(M\backslash\{Q\})$ becomes a graph of the function 
$$\psi(y)=\frac{f(x)}{|x|^2+f^2(x)}.$$
Here the coordinates $x$ and $y$ are related by the formula \eqref{eqn:AsymCord_new}.
Then under the coordinates $y$, ${\hat{g}} = \rho^{-2} g=|\ud y|^2+\ud \psi \otimes \ud \psi$ is the induced metric, thereby  $\hat{g}_{\alpha \beta} = \delta_{\alpha \beta} + \psi_\alpha \psi_\beta $.

Note that
$$\psi(y)=\frac{\frac{\lambda}{2}\phi(x)+O(|x|^{\frac{n}{2}})}{\phi(x)+O(|x|^{\frac{n}{2}+2})}=\frac{\lambda}{2}+O(|x|^{\frac{n}{2}-2})=\frac{\lambda}{2}+O(|y|^{-\frac{n}{2}+2}).$$
It is direct to check that
$$|\partial \psi|+|y|\cdot|\partial^2 \psi|+|y|^2\cdot|\partial^3 \psi|=O(|y|^{-\frac{n}{2}+1}).$$
As in Lam's work \cite{Lam}, we have
\begin{align*}
m(\hat{g}) &= \frac{1}{|\Sp^{n-1}|} \lim_{t \to \infty} \int_{\Sp^{n-1}_t} (\partial_\alpha \hat{g}_{\alpha\beta} -\partial_\beta \hat{g}_{\alpha\alpha}) \nu^\beta \ud \sigma_{\Sp_t^{n-1}} \\
&= \frac{1}{|\Sp^{n-1}|} \lim_{t \to \infty} \int_{\Sp^{n-1}_t} (\psi_{\alpha \alpha} \psi_\beta - \psi_{\alpha \beta} \psi_\alpha)  \nu^\beta \ud \sigma_{\Sp_t^{n-1}} .
\end{align*}
However, since $I_Q(M\backslash\{Q\})$ may not be an entire graph, it seems impossible to write this directly as a global integration as in \cite{Lam}. To overcome this difficulty, we discover a new mass formula that works for a general asymptotically flat hypersurface, which is a graph only at its end. 

The key observation is that the vector field 
$$\frac{1}{1+|\nabla\psi|^2}(\psi_{\alpha\alpha}\psi_\beta-\psi_{\alpha\beta}\psi_\alpha)\frac{\partial}{\partial y_\beta}$$
is in fact globally defined on the hypersurface, although it is not necessarily an entire graph, and its divergence equals a smooth function multiple of the scalar curvature. This can be verified directly using the fact that the hypersurface can be realized as a level set of some smooth function $\mathcal F$ on $\RR^{n+1}$, and each term above can be expressed by $\mathcal F$ and its derivatives. 
Instead, we  switch to the following simpler and more geometric approach here.

We start with an elementary lemma regarding hypersurfaces. 

\begin{proposition}\label{prop:useful_formula}
Let $M\subset \RR^{n+1}$ be an oriented hypersurface with a unit normal vector field $\bf n$. Let $\bf e$ be a constant vector field on $\RR^{n+1}$ and let $v:=\bold e^\top$ be its tangential part along $M$. For a symmetric 2-tensor $T$, consider $X:=(i_v T)^\sharp$ , the dual of the 1-form $i_vT$ or raising index of $i_vT$, then  
$$\div_g X=(\div_g T)(v)+\langle\bold e,\bold n\rangle\langle II, T\rangle.$$
In particular, if $\div_g T=0$, then 
$$\div_g X=\langle\bold e,\bold n\rangle\langle II, T\rangle.$$
\end{proposition}

\begin{proof}
Let $\{e_\alpha\}_{\alpha=1}^n$ be a local orthonormal frame along $M$. First we compute $\nabla^g_{e_\alpha} v$:
\begin{align*}
    \nabla^g_{e_\alpha} v &=(\nabla_{e_\alpha} v)^\top\\
    &= \Big(\nabla_{e_\alpha}(\bold e-\langle\bold e,\bold n\rangle \bold n)\Big)^\top\\
    &=\Big(-e_\alpha\langle\bold e,\bold n\rangle \bold n-\langle\bold e,\bold n\rangle \nabla_{e_\alpha}\bold n\Big)^\top\\
    &=\langle\bold e,\bold n\rangle h_{\alpha\beta}e_\beta.
\end{align*}
 This follows that\footnote{Similar computations also appear in  \cite[Proposition 2.6]{WWX} for Newton transforms.} 
\begin{align*}
    \div_g X&=\sum_\alpha \langle\nabla^g_{e_\alpha} X, e_\alpha\rangle\\
    &=\sum_\alpha \Big(\nabla^g_{e_\alpha}(i_v T) \Big)(e_\alpha)\\
    &=\sum_\alpha \Big(e_\alpha T(v,e_\alpha)-T(v,\nabla^g_{e_\alpha} e_\alpha)\Big)\\
    &=\sum_\alpha \Big((\nabla^g_{e_\alpha} T)(v,e_\alpha)+T(\nabla^g_{e_\alpha} v,e_\alpha)\Big)\\
    &=(\div_g T)(v)+\langle\bold e,\bold n\rangle\sum_{\alpha,\beta} h_{\alpha\beta}T(e_\beta,e_\alpha)\\
    &=(\div_g T)(v)+\langle\bold e,\bold n\rangle\langle II, T\rangle.
\end{align*}  
\end{proof}

Now we are ready to prove Theorem \ref{thm:main_massFormula}.

\begin{proof}[Proof of Theorem \ref{thm:main_massFormula}.]
Recall that $N$ is a complete non-compact hypersurface in $\RR^{n+1}$ with induced metric $g$. And there is a compact subset $K\subset N$ such that $N\setminus K$ is the graph of a smooth function $\psi: \RR^n\setminus B_R(0)\to\RR$, satisfying   $$|\partial_\alpha\psi|+|y||\partial^2_{\alpha\beta}\psi|+|y|^2|\partial^3_{\alpha\beta\gamma}\psi|=O(|y|^{-\frac{\tau}{2}})$$
    with $\tau>
    \frac{n-2}{2}$. 
    By the Codazzi equation, $II-Hg$ is divergence-free\footnote{This is essentially the first order Newton transform for the hypersurface.}. From now on, we choose $\bold e=e_{n+1}$, $T=H g-II$ such that
$$\div_{ g} X=\langle\bold e,\bold n\rangle (H^2-|II|^2)=\langle\bold e,\bold n\rangle R_{ g}.$$

\begin{claim}
   We have
$$i_X \ud V_{g}=X \lrcorner \ud V_g =\frac{1}{1+|\nabla\psi|^2}(\psi_{\alpha \alpha} \psi_\beta - \psi_{\alpha \beta} \psi_\alpha)  \nu^\beta \ud \sigma_{\Sp_t^{n-1}}.$$  
\end{claim}

In fact, along $N\setminus K$, we identify $\frac{\partial}{\partial y_\alpha}$ with $e_\alpha+\psi_\alpha e_{n+1}$. So we have
$$\bold e^\top=e_{n+1}-\frac{(-\nabla\psi, 1)}{1+|\nabla\psi|^2}=\frac{\psi_\alpha e_\alpha+|\nabla\psi|^2 e_{n+1}}{1+|\nabla\psi|^2}=\frac{\psi_\alpha}{1+|\nabla\psi|^2}\frac{\partial}{\partial y_\alpha}.$$
Then we obtain
\begin{align*}
   \Big(i_{\bold e^\top} II\Big)^\sharp &=\Big(\frac{\psi_{\alpha\beta}\psi_\alpha \ud y_\beta}{(1+|\nabla\psi|^2)^{\frac{3}{2}}}\Big)^\sharp\\
   &=\Big(\delta_{\gamma\beta}-\frac{\psi_\gamma\psi_\beta}{1+|\nabla\psi|^2}\Big)\frac{\psi_{\alpha\beta}\psi_\alpha}{(1+|\nabla\psi|^2)^{\frac{3}{2}}}\frac{\partial}{\partial y_\gamma}\\
   &=\Big(\frac{\psi_{\alpha\gamma}\psi_\alpha}{(1+|\nabla\psi|^2)^{\frac{3}{2}}}-\frac{\psi_\gamma\psi_{\alpha\beta}\psi_\alpha\psi_\beta}{(1+|\nabla\psi|^2)^{\frac{5}{2}}}\Big)\frac{\partial}{\partial y_\gamma}
\end{align*}
and
\begin{align*}
   \Big(i_{\bold e^\top}(Hg) \Big)^\sharp &=\Big(\frac{H\psi_\alpha}{1+|\nabla\psi|^2}(\delta_{\alpha\beta}+\psi_\alpha\psi_\beta)\ud y_\beta\Big)^\sharp=(H\psi_\beta \ud y_\beta)^\sharp\\
   &=H\psi_\beta\Big(\delta_{\gamma\beta}-\frac{\psi_\gamma\psi_\beta}{1+|\nabla\psi|^2}\Big)\frac{\partial}{\partial y_\gamma}=\frac{H\psi_\gamma}{1+|\nabla\psi|^2}\frac{\partial}{\partial y_\gamma}\\
   &=\frac{\psi_\gamma}{1+|\nabla\psi|^2}\Big(\frac{\Delta\psi}{\sqrt{1+|\nabla\psi|^2}}-\frac{\psi_{\mu\nu}\psi_\mu\psi_\nu}{(1+|\nabla\psi|^2)^{\frac{3}{2}}}\Big)\frac{\partial}{\partial y_\gamma}.
\end{align*}

Consequently, 
$$X=\Big(i_{\bold e^\top}(Hg-II)\Big)^\sharp=\frac{1}{(1+|\nabla\psi|^2)^{\frac{3}{2}}}(\psi_{\alpha\alpha}\psi_\beta-\psi_{\alpha\beta}\psi_\alpha)\frac{\partial}{\partial y_\beta}.$$
Thus, we also have
$$i_X \ud V_g=i_X\Big(\sqrt{1+|\nabla\psi|^2} \ud V_{\RR^n}\Big)=\frac{1}{1+|\nabla\psi|^2}(\psi_{\alpha \alpha} \psi_\beta - \psi_{\alpha \beta} \psi_\alpha)  \nu^\beta \ud \sigma_{\Sp_t^{n-1}}.$$

Using this and Proposition \ref{prop:useful_formula} we conclude that
\begin{align*}
m(g) &= \frac{1}{|\Sp^{n-1}|} \lim_{t \to \infty} \int_{\Sp^{n-1}_t} (\psi_{\alpha \alpha} \psi_\beta - \psi_{\alpha \beta} \psi_\alpha)  \nu^\beta \ud \sigma_{\Sp_t^{n-1}}  \\
&= \frac{1}{|\Sp^{n-1}|} \lim_{t \to \infty} \int_{\Sp^{n-1}_t} X \lrcorner \ud V_g \\
&= \frac{1}{|\Sp^{n-1}|}\int_{N} \div_g (X) \ud V_g \\
&=\frac{1}{|\Sp^{n-1}|}\int_{N}\langle e_{n+1},\bold n\rangle R_g \ud V_g.
\end{align*}
This completes the proof of Theorem \ref{thm:main_massFormula}.\end{proof}

As a direct application of Theorem \ref{thm:main_massFormula}, we give a unified proof of $m(\hat g)=0$ when $n\geq 3$, including the exceptional case $n=4k, k\geq 2$.

Now we can prove Theorem \ref{thm:main_conf_Laplace}.

\begin{proof}[Proof of Theorem \ref{thm:main_conf_Laplace}.]
Since we now have Theorems \ref{thm:main_ADM_mass} and \ref{thm:main_massFormula}, the remaining part of the proof is identical to that of \cite[Theorem 4]{Chen-Shi}. So we only outline the proof and refer the interested readers to \cite{Chen-Shi} for details.

As in \cite[Lemma 5]{Chen-Shi}, we have already proved that if there is a point $Q\in M$ such that the Green function for the conformal Laplacian has the form of $G(P,Q)=\frac{1}{(n-2)|\Sp^{n-1}|}\|P-Q\|^{2-n},\forall~ P\in M\setminus\{Q\}$, then $Q$ is umbilical. Also, the Green function for the conformal Laplacian implies that $R_{\hat g}\equiv 0$. Then we deduce from Theorems \ref{thm:main_ADM_mass} and \ref{thm:main_massFormula} that the mass of $\hat g$ is zero. Since our manifold has an induced spin structure from $\RR^{n+1}$, we can apply the spin version of the Positive Mass Theorem in \cite{bartnik} to conclude that $(M\setminus\{Q\},\hat g)$ is isometric to the flat $\RR^n$. Then the theorem follows from the following lemma, which is contained in the proof  in \cite{Chen-Shi} of \emph{\textbf{GFRC}} when $n=3,4,5$. 
\end{proof}

\begin{lemma}\label{lem:finalStep}
Let $M\subset\RR^{n+1}$ be an embedded closed hypersurface, and $Q\in M$ is an umbilical point. Let $I_Q$ be the inversion of $\RR^{n+1}$ with center $Q$ as in \eqref{def:inversion}. If $I_Q(M\setminus \{Q\})$ with induced metric $\hat g$ is isometric to the flat space form $\RR^n$, then $M$ is a round sphere.   
\end{lemma}

For the convenience of readers, we include a heuristic and sketched proof.

\begin{proof}
Regarding the isometry as the identity map and $Q$ as the origin, we obtain a smooth embedding $F: \RR^n \to \RR^{n+1}$ such that
$$g=F^*g_{\RR^{n+1}}=|F|^4 g_{\RR^n} \qquad \mathrm{~~in}\quad \RR^n$$
with the normalization that $\lim_{|y|\to\infty}F(y)=0$. Consider an inverted hypersurface $U$
\begin{align*}
w:=U(y)=I_Q \circ F(y), \qquad y \in \RR^n.
\end{align*}
A nice observation
$$U^\ast g_{\RR^{n+1}}=F^\ast I_Q^\ast g_{\RR^{n+1}}=F^\ast(|w|^{-4} |\ud w|^2)=|F(y)|^{-4} g=|\ud y|^2$$
implies that $U$ is a complete flat hypersurface in $\RR^{n+1}$.

Using Hartman-Nirenberg's characterization in \cite{Hartman-Nirenberg} of complete flat hypersurfaces in $\RR^{n+1}$ , we know that $(I_Q(M\setminus\{Q\}),\hat g)$ is a generalized cylinder in $\RR^{n+1}$ and can be parametrized as
$$U(t,z)=(x(t),y(t),z), \qquad t \in \RR,~~z \in \RR^{n-1},$$
where $t \mapsto (x(t),y(t))$ is a regular plane curve parametrized by arc length $t$. Consequently, $M$ can be parametrized by
$$F(t,z)=\frac{(x(t),y(t),z)}{x(t)^2+y(t)^2+|z|^2}.$$
Using this special parametrization of $M$, we can compute the principal curvatures of $M$ in $\RR^{n+1}$. It turns out that there are only two possible values for the principal curvatures at any point. The first one, $$\lambda=-2(xy'-x'y)$$ is a principal curvature of multiplicity at least $n-1$, and the other one is
$$\mu=\lambda-k(t)(x^2+y^2+|z|^2),$$ where $k(t)$ is the relative curvature of the plane curve $(x(t),y(t))$.

Since $Q$ is umbilical, $\lambda-\mu$ tends to zero when the point approaches $Q$. This forces $k(t)$ to vanish everywhere and hence $I_Q(M\setminus\{Q\})\subset \RR^{n+1}$ must be a hyperplane. As a consequence, $M$ is itself a round sphere.
\end{proof}

\section{Positive energy theorem and GFRC for $P_4^g$}\label{Sect:Paneitz}

Now we turn to the \emph{\textbf{GFRC}}  for Paneitz operator $P_4^g$. Similar to the conformal Laplacian case, the main tool is the positive energy theorem for $P_4^g$.

For $n \geq 5$, let $(N^n,\hat g)$ be an asymptotically flat manifold  of order $\tau>\frac{n-4}{2}$ and regularity order 4. If $Q_{\hat g} \in L^1(N,\hat g)$, then a delicate calculation together with the definition \eqref{def:Q-curv} of $Q$-curvature yields (cf. \cite[p.8]{ALL})
\begin{equation}\label{Q-curv:AF}
Q_{\hat g}=-\frac{1}{2(n-1)}\Delta_{\hat g} R_{\hat g}+O(t^{-2\tau-4})
\end{equation}
under asymptotic coordinates $y$ with $t=|y|$, and thus $\mathsf{m}_4(\hat g)$ is well defined. Notice that
$$\nu_{\hat g}=\frac{{\hat g}^{\alpha \beta}y_\beta \pa_\alpha}{\|y\|_{\hat g}}\qquad \mathrm{~~with~~}\quad \|y\|_{\hat g}:=\sqrt{{\hat g}^{\alpha \beta}y_\alpha y_\beta}$$ 
is the outward unit normal vector field on $\Sp_t^{n-1}$ with respect to $\hat g$, thereby
$$|\nu_{\hat g} -\nu|=O(t^{-\tau}).$$
Based on these facts the authors \cite[Proposition 2]{ALL} established another equivalent formula of geometric flavor:
\begin{align*}
\mathsf{m}_4(\hat g)=-\frac{1}{|\Sp^{n-1}|}\lim_{t \to \infty} \int_{\Sp_t^{n-1}}\pa_t R_{\hat g} \ud \sigma_{\Sp_t^{n-1}}.
\end{align*}
In our setting, we prefer to use the latter formula.

To continue, we introduce the Yamabe constant of an AF manifold $(N^n, g), n\geq 3$ by
\begin{align*}
Y(N,[\hat g]):= \inf_{ u \in C_c^\infty(N)\backslash\{0\}}\frac{\int_N \frac{4(n-1)}{n-2}|\nabla_g  u|^2+R_g u^2 \ud V_g}{(\int_{N} |u|^{\frac{2n}{n-2}} \ud V_g)^{\frac{n-2}{n}}}.
\end{align*}
In \cite{ALL}, the authors proved the following fourth-order positive energy theorem.

\begin{theoremletter}[\protect{\cite[Theorem 4]{ALL}}]\label{PMT:fourth-order}
Let $(N^n,g)$  be an AF manifold of asymptotic order $\tau>\frac{n-4}{2}$ and regularity order $\ell=4$, with $n \geq 5$. Assume that (a) $Q_g \geq 0$ and $Y(N,[g])>0$; (b) $Q_g \in L^1(N,g)$. Then $\mathsf{m}_4(g)\geq 0$, with equality if and only if $(N,g)$ is isometric to the flat Euclidean space $\Rn$.
\end{theoremletter}


To prove Theorem \ref{thm:main_4th_mass}, let $M\subset\RR^{n+1}$ be a cloded embedded hypersurface and $Q\in M$ is umbilical. As before, let $\hat M=M\backslash \{Q\}$ and $\hat g=\rho^{-2}g$. Then it is in fact the induced Riemannian metric on $I_Q(M\setminus Q)\subset\RR^{n+1}$. 

\begin{proof}[Proof of Theorem \ref{thm:main_4th_mass}.]
As in Section \ref{Sect:Asym_coor_mass}, write $M$ as a graph of a smooth function $f$ near $Q$. If $f(x) = \phi(x) + A_k + O(r^{k+1})$ for $k \geq 3$, then the calculations in Section \ref{Sect:Asym_coor_mass} indicate that
\begin{equation}\label{order:AF_mfld}
\hat g=\rho^{-2} g= \rho^{-2} (|\ud x|^2+\ud f \otimes \ud f)=|\ud y|^2+O(|y|^{-k}).
\end{equation}
We use the same \emph{adapted} asymptotic coordinates as in \eqref{coor:asymp_natural} and \eqref{inverse_coor:asymp_natural}: 
$$y = \frac{\lambda}{2}\frac{x}{\phi(x)} \quad \Longleftrightarrow \quad x = \frac{y}{\frac{\lambda^2}{4}+\vert y \vert^2},$$
and agian denote $r=|x|, t=|y|$. 
By the inductive estimates \eqref{order_R:induction} and \eqref{eqn:scalar_curv} we know that
$$
R_{\hat{g}} = O(r^{2k}) = O(t^{-2k}).
$$
Observe that 
\begin{align*}
\Delta_{\hat{g}} R_{\hat{g}} &= \hat{g}^{\alpha\beta} \frac{\pa^2 R_{\hat{g}}}{\pa y_\alpha \pa y_\beta} - \hat{g}^{\alpha \beta} \hat{\Gamma}_{\alpha \beta}^\gamma \frac{\pa R_{\hat{g}}}{\pa y_\gamma} \\
&= (\delta_{\alpha\beta} + O(t^{-k}))\frac{\pa^2 R_{\hat{g}}}{\pa y_\alpha \pa y_\beta} + \hat{g}^{\alpha \beta} O(t^{-k-1})O(t^{-2k-1}) \\
&= \sum_\alpha \frac{\pa^2 R_{\hat{g}}}{\pa y_\alpha^2} + O(t^{-3k-2}).
\end{align*}
Here we also used Remark \ref{rem:derivative_decay}.  To adopt the Einstein summation convention, we  write the main term as $\frac{\pa^2 R_{\hat{g}}}{\pa y_\alpha^2}$.
A direct computation yields
\begin{align*}
\frac{\pa x_\beta}{\pa y_\alpha}\frac{\pa x_\gamma}{\pa y_\alpha} =& \frac{\delta_{\beta\gamma}}{(\frac{\lambda^2}{4}+t^2)^2} - \frac{\lambda^2 y_\beta y_\gamma}{(\frac{\lambda^2}{4}+t^2)^4}\\
=& t^{-4}\delta_{\beta \gamma} + O(t^{-6}) \\
=& r^4\delta_{\beta \gamma} + O(r^6)
\end{align*}
and
\begin{align*}
\frac{\pa^2 x_\beta}{\pa y_\alpha^2} =&-(2n+4)\frac{y^\beta}{(\frac{\lambda^2}{4}+t^2)^2}+\frac{8 t^2 y^\beta}{(\frac{\lambda^2}{4}+t^2)^3}\\
=&(4-2n)\frac{t^2 y^\beta}{(\frac{\lambda^2}{4}+t^2)^3}-\frac{n+2}{2}\frac{\lambda^2 y^\beta}{(\frac{\lambda^2}{4}+t^2)^3}\\
=&(4-2n)r^2 x_\beta + O(r^5).
\end{align*}
Hence, putting these facts together we obtain
\begin{align*}
\sum_\alpha \frac{\pa^2 R_{\hat{g}}}{\pa y_\alpha^2}
&= \frac{\pa^2 R_{\hat{g}}}{\pa x_\beta \pa x_\gamma } \frac{\pa x_\beta}{\pa y_\alpha}\frac{\pa x_\gamma}{\pa y_\alpha} + \frac{\pa R_{\hat{g}}}{ \pa x_\beta} \frac{\pa^2 x_\beta}{\pa y_\alpha^2} \\
&= r^4 \Delta R_{\hat{g}} + (4-2n) r^2 x_\alpha \pa_\alpha R_{\hat{g}} + O(r^{2k+4}).
\end{align*}
Notice that $2k+4 \leq 3k+2$ for $k \geq 3$. This in turn implies that
$$
\Delta_{\hat{g}} R_{\hat{g}} = r^4 \Delta R_{\hat{g}} + (4-2n) r^2 x_\alpha \pa_\alpha R_{\hat{g}} + O(r^{2k+4}).
$$

For simplicity, we introduce
$$
B_{2k-4}:= 4n(n-1) (k-1)^2\frac{A_k^2}{r^4} -4(n-1)(k-1) \frac{A_k}{r^2} \Delta A_k + (\Delta A_k)^2 - \vert \nabla^2 A_k \vert^2,
$$
which is a homogeneous polynomial of degree $2k-4$. Then we have 
\begin{equation}\label{eqn:R_B_2k-4}
R_{\hat{g}} = r^4 B_{2k-4} + O(r^{2k+1})
\end{equation}
again by \eqref{order_R:induction} and \eqref{eqn:scalar_curv}.
Consequently, these together with \eqref{order:AF_mfld}, \eqref{Q-curv:AF}, \eqref{def:Q-curv} lead to
\begin{align}\label{asym_expan:Q}
Q_{\hat{g}} &=-\frac{1}{2(n-1)} \Delta_{\hat{g}} R_{\hat{g}} + O(t^{-2k-4})\no \\
&= -\frac{1}{2(n-1)} (r^4 \Delta(r^4 B_{2k-4}) + (4-2n) r^2 x_\alpha \pa_\alpha (r^4 B_{2k-4}) ) + O(r^{2k+4}).
\end{align}

Now we need the following technical lemma.

\begin{lemma}\label{lem:Q_integrability}
  Suppose we have 
  $$Q_{\hat g}=q(\theta)|x|^{2\ell+2}+o(|x|^{2\ell+2}),$$ with $q(\theta)$ not identically zero. Then $Q_{\hat g} \in L^1(\hat M, \hat g)$  if and only if $\ell>\frac{n-2}{2}$. 
\end{lemma}
\begin{proof}
Since $\rho=r^2(1+o(1))$, we have
$$|Q_{\hat g}| \ud V_{\hat g}= r^{2\ell+2}(1+o(1))\rho^{-n} \ud V_g=r^{2\ell+2-2n}(1+o(1))\ud V_g.$$
Then the remaining proof is identical to that of Lemma \ref{lem:R_integrability}.
\end{proof}

Since initially $k\geq 3$, we automatically obtain that $Q_{\hat g}=q(\theta)|x|^{2\ell+2}+o(|x|^{2\ell+2})$ for some $\ell\geq 3$ by \eqref{asym_expan:Q}. Hence a direct consequence of Lemma \ref{lem:Q_integrability} is that $Q_{\hat g} \in L^1(M,\hat g)$ when $5 \leq n \leq 7$.

\medskip

Next we assume $n\geq 8$ and fix $3 \leq k \leq \frac{n-2}{2}$. Since $Q_{\hat g} \in L^1(\hat M,\hat g)$,  by Lemma \ref{lem:Q_integrability} and \eqref{asym_expan:Q}  we have
\begin{align*}
0 &= r^2 \Delta(r^4 B_{2k-4}) + (4-2n) x_\alpha \pa_\alpha (r^4 B_{2k-4}) \\
&= r^2 \left[ (\Delta r^4) B_{2k-4} + 2 \nabla r^4 \cdot \nabla B_{2k-4} + r^4 \Delta B_{2k-4} \right] + (4-2n) \cdot 2k r^4 B_{2k-4} \\
&= r^4 \left[ r^2 \Delta B_{2k-4} +(24k +4n -4kn-24) B_{2k-4}\right]\\
&=r^4 \left[ r^2 \Delta B_{2k-4} -4(n-6)(k-1) B_{2k-4}\right].
\end{align*}
This implies $r^2 \mid B_{2k-4}$. 

By induction, we assume that $B_{2k-4} = r^{2l} p(x)$ for $2l \leq 2k-4$ and $l \geq 0$.
A direct calculation yields
$$
\frac{1}{r^{2l-2}} \Delta B_{2k-4} \equiv 2l(n+4k-2l-10) p(x) \quad \mod r^2
$$
and then
\begin{align*}
0 =& \frac{1}{r^{2l}} \left[ r^2 \Delta B_{2k-4} -4(n-6)(k-1) B_{2k-4}\right] \\
\equiv& \left[ 2l(n+4k-2l-10) -4(n-6)(k-1)\right]p(x)\mod r^2 \\
\equiv& -2(2k-2-l)(n-6-2l) p(x) \mod r^2.
\end{align*}

 We first notice that  $2k-2-l \geq l+2\geq 2$ as $2l \leq 2k-4$.  

When $2l \leq 2k-6$,  $n -2l -6 \geq n -2k  > 0$, then $r^2 \mid p(x)$. So, $r^{2l+2} \mid B_{2k-4}$. In particular  $r^{2k-4} \mid B_{2k-4}$ implies $B_{2k-4}=c r^{2k-4}$ for some $c \in \RR$.  When $2l = 2k-4$, we already have  $B_{2k-4}=c r^{2k-4}$. Consequently, in all cases we have $B_{2k-4}=c r^{2k-4}$.

\medskip

For clarity we let $A_k (x) = r^{2i} p(x)$, $0\leq i \leq \frac{k}{2}$, and we distinguish two cases:
\begin{enumerate}
\item[(1)]
If $n = 2k + 2$, then we use the same calculation as in \eqref{cal:coeff_R} together with the definition of $B_{2k-4}$ to conclude that
\begin{align}\label{indunction:4th-mass}
c r^{n-4i-2} =& 4n(n-1)(k-1)^2 \frac{A_k^2}{r^{4i}} - 4(n-1)(k-1)\frac{A_k}{r^{2i}} \frac{\Delta A_k}{r^{2i-2}} + \frac{1}{r^{4i-4}} \left[ (\Delta A_k)^2 - \vert \nabla^2 A_k \vert^2 \right] \no\\
\equiv& 4(k-i-1) \left[ (k-i-1) (n-4i)(n-1) + 2i(k-2i)\right] \tilde p^2 (x) \quad  \mod r^2.
\end{align}
Notice that the coefficient of $\tilde p^2(x)$ is nonzero. When $i<\frac{k}{2}$, we have $r^2 \mid \tilde p^2(x)$ and hence $r^2 \mid \tilde p(x)$, that is, $r^{2i+2} \mid A_k$. To continue, we further distinguish it by two cases.
\begin{enumerate}[(i)]
\item If $k = 2j -1$ for $j \geq 2$ and $n = 4j$, then we apply \eqref{indunction:4th-mass} with $i=j-1$ to show that $r^{2j} \mid A_k$, which implies $A_k = 0$.

\item If $k = 2j$ for $j \geq 2$ and $n = 4j + 2$, then letting $i=j-1$ in \eqref{indunction:4th-mass} we obtain
$$4 j [6(n-1)j+4(j-1)]\tilde p^2(x) \equiv 0 \quad  \mod r^2.$$
This implies $r^2 \mid \tilde p^2(x)$ and then $r^2 \mid \tilde p(x)$. Eventually, we have
$A_k = \tilde{\kappa} r^{2j}=\tilde{\kappa} r^{\frac{n-2}{2}}$ for some $\tilde{\kappa} \in \RR$.
\end{enumerate}

\item[(2)] If $n \neq 2k+2$, then $n -2l -6 \neq 0$. So, the same argument above gives $r^{2k-2} \mid B_{2k-4}$ and then $B_{2k-4}=0$.  This together with \eqref{indunction:4th-mass} yields $\tilde p^2(x)\equiv 0 \mod r^2$, which implies $r^2 \mid \tilde p(x)$ and then $r^{2i+2} \mid A_k$. So, we repeat the procedure up to $k=\lfloor \frac{n}{2} \rfloor-1$ and obtain $r^{\lfloor \frac{n}{2}\rfloor}  \mid A_k=0$, which implies $A_k=0$.
\end{enumerate}

Therefore, collecting these facts together we arrive at

\begin{proposition}\label{prop:vanishing_order_g-Q}
If $n\geq 8$ and $Q_{\hat{g}} \in L^1(M,\hat{g})$, then near $Q$  we have
\begin{enumerate}[(i)]
\item when $n \neq 4j+2$ for $j\geq 2$, $f(x) = \phi (x) + O(|x|^{\lfloor \frac{n}{2} \rfloor})$;
\item when $n=4j+2$ for $j\geq 2$, $f(x) = \phi(x) + \tilde{\kappa} |x|^{\frac{n}{2}-1} + O(|x|^{\frac{n}{2}}), ~\tilde{\kappa} \in \RR$.
\end{enumerate}
\end{proposition}

To complete the proof of Theorem \ref{thm:main_4th_mass}, we only need to determine the asymptotic order and to explicitly compute the fourth-order energy.

\begin{enumerate}
\item[(1)] When $5\leq n\leq 7$, then $f(x) = \phi (x) + O(|x|^3)$, and Lemma \ref{lem:asym_order-coor} yields that $(M\backslash \{Q\},\hat g)$ is asymptotically flat of order 3. Note that we do not need to worry about the regularity order in view of Remark \ref{rem:derivative_decay}. Since $R_{\hat g}=O(t^{-6})$, we have
\begin{align*}
\mathsf{m}_4(\hat{g})=&- \frac{1}{\vert \mathbb{S}^{n-1} \vert} \lim_{t \to \infty} \int_{\mathbb{S}_t^{n-1}} \pa_t R_{\hat{g}} \ud \sigma_{\Sp_t^{n-1}} \\
=& \frac{1}{\vert \mathbb{S}^{n-1} \vert} \lim_{t \to \infty} \int_{\mathbb{S}^{n-1}} O(t^{n-8}) \ud \sigma \\
=&0.
\end{align*}
\item[(2)]
When $n \neq 4j +2$ for $j \geq 2$, by Proposition \ref{prop:vanishing_order_g-Q}, \eqref{order_R:induction} and \eqref{eqn:scalar_curv} we obtain
$$
R_{\hat{g}} = O(t^{-2m}) \qquad \mathrm{~~with~~}\quad m =\lfloor \frac{n}{2} \rfloor.
$$
Combing Proposition \ref{prop:vanishing_order_g-Q} and Lemma \ref{lem:asym_order-coor} yields that $(M\backslash \{Q\},\hat g)$ is asymptotically flat of order $\lfloor \frac{n}{2} \rfloor$. 
Hence, it is direct to calculate
\begin{align*}
\mathsf{m}_4(\hat{g})=&- \frac{1}{\vert \mathbb{S}^{n-1} \vert} \lim_{t \to \infty} \int_{\mathbb{S}_t^{n-1}} \pa_t R_{\hat{g}} \ud \sigma_{\Sp_t^{n-1}} \\
=& \frac{1}{\vert \mathbb{S}^{n-1} \vert} \lim_{t \to \infty} \int_{\mathbb{S}^{n-1}} O(t^{n-2m-2}) \ud \sigma \\
=&0.
\end{align*}

\item[(3)]When $n = 4j+2$,  letting $i=\frac{n-2}{4}$ in \eqref{indunction:4th-mass} we obtain
$$c=\frac{(n-1)(n-6)^2}{2} \tilde \kappa^2,$$
whence
$$
R_{\hat{g}} = c t^{2-n} + O(t^{1-n})
$$
by \eqref{eqn:R_B_2k-4}. Similarly, Lemma \ref{lem:asym_order-coor} together with Proposition \ref{prop:vanishing_order_g-Q} yields that $(M\backslash \{Q\},\hat g)$ is asymptotically flat of order $\frac{n}{2}-1$. Then we have
\begin{align*}
\mathsf{m}_4(\hat{g})&= -\frac{1}{\vert \mathbb{S}^{n-1} \vert} \lim_{t \to \infty} \int_{\mathbb{S}_t^{n-1}} \pa_t R_{\hat{g}} \ud \sigma_{\Sp_t^{n-1}}\\
&= -\frac{1}{\vert \mathbb{S}^{n-1} \vert} \lim_{t \to \infty} \int_{\mathbb{S}^{n-1}} ((2-n)c t^{1-n} + O(t^{-n})) t^{n-1}\ud \sigma \\
&=\frac{(n-1)(n-2)(n-6)^2}{2} \tilde \kappa^2.
\end{align*}
\end{enumerate}
Hence the fourth-order energy is always non-negative and vanishes except for the case $n=4j+2, j\geq 2$. This concludes the proof.
\end{proof}

Now we use Theorem \ref{thm:main_4th_mass} to prove Theorem \ref{thm:Paneitz_oper}. 
\begin{proof}[Proof of Theorem \ref{thm:Paneitz_oper}.]
For $n\geq 5$ and $n\neq 4j+2, j\geq 2$, let $(M^n,g)$ be a closed embedded hypersurface in $\RR^{n+1}$ such that the Yamabe constant $Y(M,[g])>0$. Assume that for some point $Q$, the Green function for $P_4^g$ with pole $Q$ is of the form $c_{n,2}\|\cdot-Q\|^{4-n}$. As before, let $\hat M=M\backslash \{Q\}$ and $\hat g=\rho^{-2}g$.  Then by definition of the Green function for $P_4^g$, we know that $Q_{\hat g}\equiv 0$ and hence integrable. Finally, by Theorem \ref{thm:main_4th_mass}, we have $\mathsf{m}_4(g)=0$

To apply the positive energy Theorem \ref{PMT:fourth-order}, we need to show that $Y(\hat M,\hat g)>0$.

\begin{lemma}\label{lem:Yamabe_const}
Under the assumption that $Y(M,g)>0$, we also have $Y(\hat M,\hat g)>0$.
\end{lemma}

\begin{proof}
For clarity, we set $L_g=P_2^g$.  Since $Y(M,[g])>0$, we let
\begin{align*}
\bar g=&G_{L_g}(\cdot,Q)^{\frac{4}{n-2}} g\\
=&G_{L_g}(\cdot,Q)^{\frac{4}{n-2}}c_{n,2}^{-\frac{4}{n-4}}\rho^2  \hat g\\
=&\left(c_{n,2}^{-1}\rho^{\frac{n-2}{2}}G_{L_g}(\cdot,Q)\right)^{\frac{4}{n-2}}\hat g.
\end{align*}
Clearly, $R_{\bar g}=0$ on $M \backslash \{Q\}$. Notice that
$$\rho^{2-n}=r^{2-n}(1+O(r^2))$$
by \cite[Lemma 9]{Chen-Shi} and
$$c_{n,2}^{-1}G_{L_g}(\cdot,Q)=r^{2-n}(1+O(r^2))$$
by Parker-Rosenberg
\cite[Theorem 2.2 and (4.4)]{ParkerRosenberg}, where $r=\ud_g(\cdot,Q)$.
Take
$$\phi=c_{n,1}^{-1}\rho^{\frac{n-2}{2}}G_{L_g}(\cdot,Q)=1+O(r^2),$$
such that $\phi-1 \in W_{-\tau}^{2,p}(M\backslash \{Q\})$\footnote{See \cite[p.566]{Maxwell} for the definition of $W_{-\tau}^{2,p}(M\backslash \{Q\})$.} for some $p>n/2$ and all $\tau \in (0,2)$, and $R_{\phi^{4/(n-2)}\hat g}=0$ on $M \backslash \{Q\}$. 
Thanks to Maxwell \cite[Proposition 3 and Remark 1]{Maxwell}, we obtain $Y(\hat M,[\hat g])>0$.
\end{proof}

Now  the positive energy Theorem \ref{PMT:fourth-order} states that $(M\backslash \{Q\},\hat g)$ is isometric to  $(\RR^n,\delta)$. Consequently, by Lemma \ref{lem:finalStep}, $M$ is itself a round sphere.

\medskip

In the remaining case $n=3$, we notice that $Y(M,g)>0$ by assumption and $G(\cdot,Q)=c_{n,2}\rho^{1/2}\leq 0$. 
Now we use the following conformal positive mass theorem due to Hang-Yang.\footnote{Since Hang-Yang's proof also hinges on the Green function for the conformal Laplacian, the assumption $Y(M,[g])>0$ is implicitly needed.}
\begin{propositionletter}[\protect{\cite[Propositon 2.4]{HY2}}]
Let $(M^3,g)$ be a close manifold such that $Y(M,[g]>0)$ and $\mathrm{ker} P_4^g=\{0\}$. If the Green function of $P_4^g$ satisfies $G(\cdot,Q)<0$ for some $Q \in M$, then $G(Q,Q)<0$ unless $(M,g)$ is conformally equivalent to $(\Sp^3,g_{\Sp^3})$.
\end{propositionletter}
Then it follows that $(M,g)$ is conformal to $\Sp^3$.
\end{proof}

\begin{remark}
When $5\leq n\leq 7$, since $Y(M,[g])>0$ by assumption, together with Proposition \ref{prop:4-th_mass_Green-fcn} we can adapt the approach of Gursky-Malchiodi \cite[Theorem 2.9]{GM} to show that $(M \backslash \{Q\},G(\cdot, Q)^{\frac{4}{n-4}}g)$ is isometric to the flat Euclidean space $\RR^n$. This gives an alternative proof in the lower dimensions.
\end{remark}

\appendix

\section{A generalization of Reilly formula for entire graphs} 

As another application of Proposition \ref{prop:useful_formula}, we generalize the result due to Reilly \cite[Proposition 4.1]{Reilly} for $\sigma_k$ curvature of graphs in Euclidean spaces, which should be of independent interest. Although Reilly's formula is originally meaningful only for graphs, the following proposition shows that it is not the case.

\begin{proposition}
Let $M\subset \RR^{n+1}$ be an oriented hypersurface with a unit normal vector field $\bf n$. Let $\bf e$ be a constant vector field on $\RR^{n+1}$ and let $v:=\bold e^\top$ be its tangential part along $M$. For $1 \leq k \leq n$, denote by  $\sigma_k(\mathcal S)$ the $k$-th elementary symmetric function of the shape operator $\mathcal S=-\nabla \bold n$ and by
$$T_{k-1}(\mathcal S)_\alpha^\beta=\frac{\pa \sigma_k(\mathcal S)}{\pa h_\beta^\alpha}$$
its $(k-1)$-th Newton transform.
Then we have
\begin{equation}\label{sigma_k-divergence}
k \sigma_k(\mathcal S)\langle\bold e,\bold n\rangle=\div_g(i_vT_{k-1}(\mathcal S))^\sharp.
\end{equation}
\end{proposition}
\begin{proof}
Since $T_{k-1}(\mathcal S)$ is divergent-free, taking the symmetric 2-tensor  $T_{\alpha \beta}=g_{\alpha \gamma}T_{k-1}(\mathcal S)_\beta^\gamma$ in Proposition \ref{prop:useful_formula} yields the desired assertion.
\end{proof}

Note that when $M$ is a graph of the form $\{(x, f(x))|\ x\in U\subset\RR^n\}$, if we take ${\bf e}=e_{n+1}$ in \eqref{sigma_k-divergence} and identify $\frac{\pa}{\pa x_\alpha}$ with $e_\alpha+f_\alpha e_{n+1}$ as above, thereby $e_{n+1}^\top=\frac{f_\alpha}{1+|\nabla f|^2}\frac{\pa}{\pa x_\alpha}$, and hence \eqref{sigma_k-divergence} recovers Reilly's original result.

\section{Mass of Green function for Paneitz operator and $\mathsf{m}_4$}
We give an alternative proof of \cite[Proposition 6]{ALL}, which characterizes the relation between the mass of the Green function $G(\cdot, Q)$ and $\mathsf{m}_4(\hat g)$ for the asymptotic manifold $(M \backslash \{Q\},\hat g:=G(\cdot, Q)^{\frac{4}{n-4}}g)$. Our proof is elementary and greatly simplifies the original proof in \cite{ALL}.

\begin{proposition}\label{prop:4-th_mass_Green-fcn}
Let $(M^n,g)$ be a closed manifold with  $n=5,6,7$ or $(M,g)$ is locally conformally flat and $n \geq 8$. Suppose $(M,g)$ admits a positive Green function $G(\cdot,Q)$ for the Paneitz operator $P_4^g$. Then there is a conformal metric of $g$ (still denoted by itself) such that
$$(c_{n,2})^{-1} G(\cdot,Q)=r^{4-n}+\mathsf{A}+O(r),\quad  r=|x|=\ud_g(x,Q),~ \mathsf{A} \in \RR,$$
and $\mathsf{m}_4(\hat g)=8(n-1)(n-2)A c_{n,2}$ for $\hat g=G(\cdot, Q)^{\frac{4}{n-4}}g$. Here, $c_{n,2}=\big(2(n-2)(n-4)|\Sp^{n-1}|\big)^{-1}$.
\end{proposition}
\begin{proof}
Under conformal normal coordinates around $Q$,  for any $N>n+3$ we have 
$$\det g=1+O(r^N), \qquad R_g=O(r^2) $$
and  from  Gursky-Malchiodi \cite[Theorem 2.9]{GM} or Hang-Yang \cite[Proposition 2.1]{HY1} that the expansion of $G(\cdot,Q)$ around $Q$ is given by
$$c_{n,2}^{-1} G(\cdot,Q)=r^{4-n}+\mathsf{A}+O(r).$$
  The readers refer to Lee-Parker \cite{Lee-Parker} for the basics of the conformal normal coordinates.

  In the following, we normalize $c_{n,2}=1$.

Consider the asymptotically flat manifold $(\hat M:=M \backslash \{Q\},\hat g=G(\cdot,Q)^{\frac{4}{n-4}}g)$.  For $n=5,6,7$, under inverted coordinates $y=x|x|^{-2}$ with $t=|y|=r^{-1}$, we have
$$\hat g_{\alpha \beta}(y)=\gamma(y)^{\frac{4}{n-4}}(\delta_{\alpha \beta}+O(t^{-2})),$$
where 
$$\gamma(y)=1+\mathsf{A} t^{4-n}+O(t^{3-n}).$$
This means that $(\hat M,\hat g)$ is asymptotically flat of order $\tau=1$ for $n=5$ and of order $\tau=2$ for $n=6,7$; if $(M,g)$ is locally conformally flat and $n\geq 8$, then we may take $g_{ij}=\delta_{ij}$ near $Q$, thereby $(\hat M,\hat g)$ is asymptotically flat of order $\tau=n-4$. So, in both cases $\mathsf{m}_4(\hat g)$ is well defined.

By definition of the fourth-order energy we have
\begin{align}\label{4th-mass}
\mathsf{m}_4(\hat g)=-\frac{1}{|\Sp^{n-1}|}\lim_{t \to \infty} \int_{\Sp_t^{n-1}}\pa_t R_{\hat g} \ud \sigma_{\Sp_t^{n-1}}
=\frac{1}{|\Sp^{n-1}|}\lim_{r \to 0} \int_{\Sp^{n-1}} r^{3-n} \pa_r R_{\hat g} \ud \sigma.
\end{align}

Write $\hat g=G^{\frac{4}{n-4}}g=(G^{\frac{n-2}{n-4}})^{\frac{4}{n-2}}g$. Then the scalar curvature equation gives
$$R_{\hat g}=G^{-\frac{n+2}{n-4}}\left(-\frac{4(n-1)}{n-2}\Delta_g G^{\frac{n-2}{n-4}}+R_g G^{\frac{n-2}{n-4}} \right).$$

For any radial function $u=u(r)$ we have
$$\Delta_g u=\Delta u+O(r^{N-1})u',$$
where $\Delta=\pa_r^2+\frac{n-1}{r}\pa_r$. Observe that
$$ \Delta_g G^{\frac{n-2}{n-4}}=\frac{n-2}{n-4} r^{1-n}\pa_r(r^{n-1}G^{\frac{2}{n-4}}G')+O(r^{N-1}) G^{\frac{2}{n-4}}|G'|=-2(n-2)\mathsf{A} r^{-4}+O(r^{-3}).$$
This follows that
$$-\frac{4(n-1)}{n-2}G^{-\frac{n+2}{n-4}}\Delta_g G^{\frac{n-2}{n-4}}=8(n-1)\mathsf{A} r^{n-2}+O(r^{n-1}).$$
Also we obtain
$$R_g G^{-\frac{4}{n-4}}=O(r^6).$$
Hence, inserting these estimates into \eqref{4th-mass} yields $\mathsf{m}_4(\hat g)=8(n-1)(n-2)\mathsf{A}$.
\end{proof}

\begin{remark}
In \cite[Proposition 2.1(c)]{Chen-Hou}, Hou and the first named author proved that there is a mass of the Green function for $P_6^g$ provided that $(M^n,g)$ is locally conformally flat and $n\geq 10$ or $n=7,8,9$.
\end{remark}

\end{document}